\documentclass[a4paper]{article}
\usepackage{amscd,amssymb,amsopn,amsmath,amsthm,graphics,amsfonts,
enumerate,verbatim,calc}
\usepackage[dvips]{graphicx}
\usepackage[colorlinks=true,linkcolor=red,citecolor=blue]{hyperref}
\usepackage[all]{xy}
\usepackage{pgf,tikz}
\usepackage{tikz-cd}
\usepackage{mathrsfs}
\usepackage{float}
\usepackage[latin1,utf8]{inputenc}
\usepackage{bbm}
\usepackage{enumitem}
\usepackage{color}
\usepackage{graphicx}
\usepackage{transparent}
\usepackage{graphics}
\usepackage{xcolor}
\usepackage{color}
\usepackage{geometry}
\usepackage{marginnote}
\usepackage{mathtools}

\usepackage{pgf,tikz}
\usepackage{mathrsfs}
\usetikzlibrary{arrows,decorations.pathmorphing,backgrounds,positioning,fit,petri,cd}

\theoremstyle{plain} 
\newtheorem{theorem}{Theorem}[section]
\newtheorem{proposition}[theorem]{Proposition}
\newtheorem{lemma}[theorem]{Lemma}
\newtheorem{corollary}[theorem]{Corollary}

\theoremstyle{remark}
\newtheorem{remark}[theorem]{Remark}
\newtheorem{example}[theorem]{Example}

\theoremstyle{definition}
\newtheorem{definition}[theorem]{Definition}

\DeclareMathOperator{\Hom}{Hom}
\DeclareMathOperator{\Ext}{Ext}

\DeclareMathOperator{\im}{Im}

\DeclareMathOperator{\Filt}{Filt}

\DeclareMathOperator{\Mod}{Mod}

\DeclareMathOperator{\Tors}{Tors}
\DeclareMathOperator{\Obj}{Obj}
\DeclareMathOperator{\tot}{tot}
\DeclareMathOperator{\WStab}{WStab}
\DeclareMathOperator{\Stab}{Stab}

\DeclareMathOperator{\Slice}{Slice}
\renewcommand{\min}{\text{min}}
\renewcommand{\max}{\text{max}}
\renewcommand{\sup}{\text{sup}}
\renewcommand{\inf}{\text{inf}}

\renewcommand{\mod}{\operatorname{mod}}

\newcommand{\A}{\mathcal{A}}

\newcommand{\X}{\mathcal{X}}
\newcommand{\Y}{\mathcal{Y}}

\newcommand{\F}{\mathcal{F}}
\renewcommand{\P}{\mathcal{P}}
\newcommand{\cP}{\mathcal{P}}
\newcommand{\T}{\mathcal{T}}
\newcommand{\fT}{\mathfrak{T}}

\newcommand{\0}{\{0\}}

\newcommand{\CT}{\mathfrak{T}}

\newcommand{\U}{\mathscr{U}}

\newcommand{\cS}{\mathcal{S}}

\newcommand{\bx}{\mathbf{x}}
\newcommand{\by}{\mathbf{y}}

\title{Weak stability conditions and the space of chains of torsion classes}
\author{Aran Tattar and Hipolito Treffinger}

\begin{document}

\maketitle

\begin{abstract}
    In this paper we show an explicit relation between chains of torsion classes and weak stability conditions over an abelian category. In particular, up to a natural equivalence, they coincide. We investigate topological properties of the space of chains of torsion classes and its quotient given by this equivalence relation. 
    In particular we show that this space is compact if and only if the abelian category has finitely many torsion classes.
\end{abstract}

\tableofcontents

\section{Introduction}

The study of stability conditions began and was pioneered by  Mumford  \cite{Mumford1965} in their work on geometric invariant theory in the 1960s. 
Since then, these notions and techniques have been adapted and generalised to much success in a variety of fields including differential geometry \cite{Tian}, triangulated and derived categories \cite{Bridgeland2007} and quiver representations \cite{King1994,Schofield1991}. 
The latter were formalised to define and study stability conditions on general abelian categories by Rudakov \cite{Rudakov1997},  which is the setting we are interested in. 
Joyce \cite{Joyce2007} then generalised these ideas further by introducing the concept of weak stability conditions for an abelian category and showed that  many important properties of stability conditions also hold for this more general class. Let us highlight of these properties in particular: 

Let $\A$ be an abelian category. 
A (weak) stability condition (satisfying mild finitness assumptions) on $\A$ determines a class of \emph{semistable} objects and a \emph{slicing} of $\A$. 
In short, this means that for each object $X \in \A$ there exists a unique ordered filtration by these semistable objects (we make this precise in Section~\ref{sec:background}). 
These filtrations, known as \emph{Harder-Narasimhan filtrations} for the work of Harder and Narasimhan \cite{Harder1975}, are a powerful tool allowing to reduce problems concerning all the objects in the abelian category to the often better behaved semistable objects.
This property is axiomatised in the concept of a \emph{slicing} (Definition \ref{def:slicing}).

On the algebraic side of the story, in \cite{Brustle2017} and \cite{T-HN-filt}, it was shown that every stability condition on $\A$ induces two  \emph{chains of torsion classes} in $\A$; we extend this to weak stability conditions in Corollary~\ref{prop:stab-chain}. 
Moreover it was shown \cite{T-HN-filt}, that every chain of torsion classes (again, satisfying mild finitness assumptions) induces Harder-Narasimhan filtrations of each object. 
We summarise this in the Figure \ref{fig:history}. 

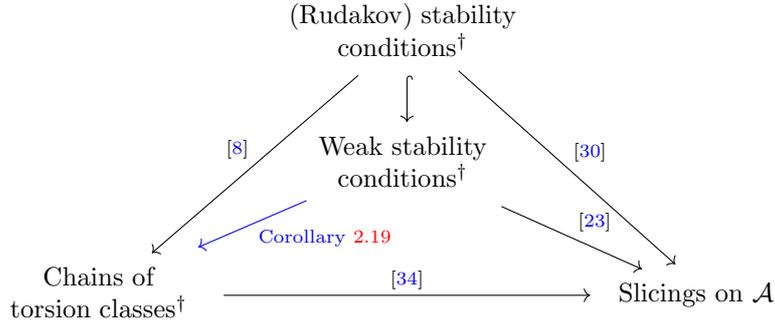
\begin{figure} \label{fig:history}
\[\begin{tikzcd}[ampersand replacement = \&, column sep = 0.5cm, every cell/.append style={align=center}]
                                                                                          \& \begin{tabular}{c}
(Rudakov) stability \\ conditions$^\dag$                                                                                          
 \end{tabular} \arrow[rdd, "{\text{\cite{Rudakov1997}}}", shift left] \arrow[d, hook] \arrow[ldd, "{\text{\cite{Brustle2017}}}"']                        \&                                        \\
 \&
\begin{tabular}{c}
Weak stability \\ conditions$^\dag$
\end{tabular}                                                                                         \arrow[ld,  blue, shift right,"{\text{Corollary \ref{prop:stab-chain}}}"] \arrow[rd, "{\text{\cite{Joyce2007}}}"] \&      
    
    \\
\begin{tabular}{c}
Chains of \\ torsion classes$^\dag$
\end{tabular} \arrow[rr, "{\text{\cite{T-HN-filt}}}"] \arrow[ru, phantom, shift right, red] \&  \&   \begin{tabular}{c}
Slicings on $\A$
 \end{tabular}                                                                                                   
\end{tikzcd} \] 
 \centering
    \caption{A pictorial summary of the known relationships between our objects of interest with a {\color{blue}generalisation}: In an ambient abelian category $\A$ the diagram (with both black and blue arrows) commutes ($\dag$: satisfying finiteness conditions). We explain all components of this diagram in Section~\ref{sec:background}.}
    \label{fig:my_label}\end{figure}

Torsion classes are an important class of subcategories in additive categories, they arose in abelian categories as an abstraction of the properties of torsion abelian groups in the category of abelian groups \cite{Dickson1966}. 
Since then, they became central in representation theory, playing a key role in tilting and $\tau$-tilting theory \cite{AIR, BB}, Auslander-Reiten theory \cite{Auslander1981}, and in connections with combinatorics \cite{Ingalls2009}, lattice theory \cite{DIRRT} and representations of posets \cite{Rognerud2021}.
We also note that the concept of a torsion class also exists in more general settings, for example quasi-abelian categories \cite{tattar2021torsion}, $n$-abelian categories \cite{Jorgensendtorsion}, and extriangulated categories \cite{adachi2022intervals}. 
Moreover, it has been shown that in the first two of these settings, that chains of torsion classes also induce Harder-Narsimhan filtrations of each object, see \cite{tattar2021torsion} and \cite{asadollahi2022higher} respectively.

One will notice that the diagram in Figure \ref{fig:history} is directed. 
It is a natural question if one may recover the data of a (weak) stability condition and/or a chain of torsion classes from the slicing it induces.
We give this question a complete answer in two steps.
First we show an explicit relationship between chains of torsion classes and weak stability conditions as follows.

\begin{theorem}[Theorems~\ref{thm:semistabchain}, \ref{thm:ohm-mho-are-wsc}  and \ref{prop:sameslicing}]
    Let $\eta$ be a chain of torsion classes in an abelian category $\A$. 
    Then $\eta$ induces two weak stability conditions on $\A$, $\Omega_\eta$ and $\mho_\eta$ having the same slicing. 
    Moreover, if $\phi$ is a weak stability condition then there exists a chain of torsion classes $\eta$ such that $\phi$, $\Omega_\eta$, and $\mho_\eta$ have the same slicing.
\end{theorem}

Moreover, these mappings when added to Figure \ref{fig:history} preserve the commutativity of the diagram (see Section~\ref{sec: order structures}). 
Consequently, up to slicing, chains of torsion classes and weak stability conditions coincide, moreover every slicing arises in this way (Proposition~\ref{prop: slice(A)  cong T(A)}). 
In other words, up to the Harder-Narasimhan filtrations they induce, the lower three vertices in Figure~\ref{fig:history} equivalent.
Afterwards, in a second instance we also give evidence that this equivalence does not extend to the the whole of Figure~\ref{fig:history} as  $\mho_\eta$ and $\Omega_\eta$ are rarely stability conditions (see Proposition \ref{prop:when-ohm-mho-stab}).

\smallskip
It follows from the ideas of Bridgeland \cite{Bridgeland2007} that class of chains of torsion classes on $\A$, $\fT(\A)$, has a natural pseudometric (Proposition~\ref{prop:metric}). 
As a quotient of $\fT(\A)$, the class of slicings $\Slice \A$ also has a natural topological structure. In the second part of the paper we devote ourselves to the study of the topological properties of these spaces.
Our main result in this part of the paper is the following.
\begin{theorem}[Theorem~\ref{thm:compact}]
    Let $\A$ be an abelian category. 
    The space $\Slice \A$ of slicings of $\A$ is compact if and only if there are finitely many torsion classes in $\A$.
\end{theorem}
The above result, when applied to module categories of finite dimensional associative algebras, gives a new characterisation of the $\tau$-tilting finite property (Corollary~\ref{cor:taufinite}).
We moreover study certain distinguished open and closed subsets of these spaces.
For instance we show that maximal green sequences, that is non-refinable finite chains of torsion classes, induce open sets in the space of chains of torsion classes (Theorem~\ref{thm:chambers}), while the set $L(X)$ of chains of torsion classes for which a fixed object $X$ is quasi-semistable is a closed set (Proposition~\ref{prop:wallII}).

\medskip

The article is organised as follows. 
In Section \ref{sec:background} we recall the notions of chains of torsion classes and weak stability conditions in abelian categories and prove some of their basic properties and explain the relationships exhibited in Figure \ref{fig:history}.
In Section \ref{sec:tors-to-wsc}, we show how each chain of torsion classes induces two canonical weak stability conditions and moreover characterise when these are stability conditions. 
We go on to show that these constructions are compatible with the commutativity of Figure \ref{fig:history} using the natural poset structures on the spaces of weak stability conditions and chains of torsion classes. 
We also see that chains of split torsion classes correspond to total weak stability conditions. 
An important consequence of this section is that when considered up to a natural equivalence relation, the classes chains of torsion classes and weak stability conditions coincide with the space of slicings, $\Slice \A$.
In the fourth section we investigate topological properties of the space of chains of torsion classes and of $\Slice \A$. 
To do this, we give an explicit construction of $\Slice \A$ and show that this space has particularly nice properties when $\Tors \A$ is finite. 
We go on to study particular open and closed sets of $\Slice \A$ and $\fT(\A)$.

\bigskip

\noindent \textbf{Conventions:} 
Recall that  a category is \textit{small} if the class of all isomorphism classes of its objects is a set.
Throughout, $\A$ will denote a small abelian category. 
Every subcategory $\X$ of $\A$ is assumed to be additive, full and closed under isomorphisms.

\smallskip
\noindent
\textbf{Acknowledgements:} 
H.~Treffinger is supported by the European Union’s Horizon 2020 research and innovation programme under the Marie Sklodowska-Curie grant agreement No 893654. The authors would like to thank Jenny August for interesting discussions. 

\section{Preliminaries} \label{sec:background}

We begin by briefly recalling some preliminary notions. 
Let $\X$ be a subcategory of $\A$.
By $\X^\perp$ we denote the subcategory of $\A$ with objects \[\X^\perp = \{ Y \in \A \mid \Hom_\A(\X, Y) = 0 \}. 
\] The subcategory ${}^\perp\X$ is defined dually.

We say that an object $M$ of $\A$ is \emph{filtered by $\X$} if there exists a finite sequence of subobjects 
$$M_0 \subset M_1 \subset \dots \subset M_n$$
such that $M_0=0$, $M_n =M$ and $M_i/M_{i-1} \in \X$ for all $1 \leq i \leq n$.
We denote $\Filt{\X}$ the category of all objects filtered by $\X$. Equivalently, $\Filt{\X}$ is the smallest subcategory of $\A$ containing $\X$ that is closed under extensions. 

An object in $\A$ is said to be of \textit{finite length} it there exists a finite filtration by the simple objects in $\A$. 
We say that $\A$ is of \textit{finite length} if every object in $\A$ is of finite length. 

\subsection{Chains of torsion classes}

Generalising the classical properties of torsion and torsionfree abelian groups, Dickson introduced in \cite{Dickson1966} the notion of \textit{torsion pair} for a abelian categories as follows.

\begin{definition}
Let $\A$ be an abelian category. 
Then a pair $(\T, \F)$ of subcategories of $\A$ is a \textit{torsion pair} if the following conditions are satisfied:
\begin{itemize}
    \item $\Hom_\A (X,Y)=0$ for all $X \in \T$ and $Y \in \F$;
    \item for every object $X$ in $\A$ there exists the a short exact sequence
$$0\to tX\to X\to fX \to 0$$
where $tX\in\mathcal{T}$ and $fX\in\mathcal{F}$.
\end{itemize}

This short exact sequence, known as the \textit{canonical short exact sequence} of $M$ with respect to $(\T, \F)$, is unique up to isomorphism.
Given a torsion pair $(\T, \F)$ we say that $\T$ is a torsion class and $\F$ is a torsionfree class. 
\end{definition}

Dickson also gave useful characterisations of torsion and torsionfree classes which we recall here.

\begin{theorem}\cite[Theorems 2.1 \& 2.3]{Dickson1966} \label{thm:dicksontorsion}
Let $\T$ and $\F$ be full additive subcategories of an abelian category $\A$. Then the following are equivalent
\begin{enumerate}
    \item $(\T, \F)$ is a torsion pair.
    \item $\T$ is closed under coproducts, quotients and extensions and $\F = \T^\perp$.
    \item $\F$ is closed under products, subojects and extensions and $\T = {}^\perp\F$.
    \item $\T = {}^\perp \F $ and $\F = \T^\perp$.
\end{enumerate}
\end{theorem}

We denote by $\text{Tors}(\A)$ the set of all torsion classes in $\A$.
Moreover, the natural inclusion of sets induces a natural partial order in $\Tors(\A)$.


Inspired by the relation between stability conditions and torsion classes, it was studied in \cite{T-HN-filt} the relation between torsion classes and Harder-Narasimhan filtrations.
This was done through the introduction of chains of torsion classes as follows.

\begin{definition}\cite[Definition 2.1]{T-HN-filt}
A chain of torsion classes $\eta$ indexed by $[0,1]$ in an abelian category $\A$ is a set of torsion classes 
$$\eta:= \{\T_s \mid \T_0=\A, \T_1=\{0\} \text{ and } \T_s \subset \T_r \text{ if } r\leq s\}.$$
Equivalently, one can think of such a chain of torsion classes as an order-reversing function $\eta: [0,1] \to \Tors \A$.
We call such an $\eta$ \emph{quasi-Noetherian} (resp. \emph{weakly-Artinian}) if for every interval $(a,b) \subset [0,1]$ and $M \in \A$ there exists $s \in (a,b)$ such that $t_r A \hookrightarrow t_sA $ (resp. $s'\in (a,b)$ such that $t_{s'}A \hookrightarrow t_rA $) for all $r \in (a,b)$ and $A \in \A$; where $t_r M$ denotes the torsion subobject of $M$ with respect to the torsion class $\T_r$. 
We denote by $\mathfrak{T}(\A)$ we denote the set of all $\eta$ that are quasi-Noetherian and weakly-Artinian.
\end{definition}

For a chain of torsion classes $\eta = (\T_i)_{i\in[0,1]}$, we write, when there is no ambiguity, $\F_i$ to denote the torsionfree class corresponding to $\T_i$. We also continue the notation from the above definition in letting $t_r M$ (resp. $f_r M$) denote the torsion subobject (resp. torsionfree quotient object) of $M \in A$ with respect to the torsion pair $(\T_r, \F_r)$ at $r \in[0,1]$.

\begin{remark}\label{rmk:WA&QN}
We make some observations regarding the weakly-Artinian and quasi-Noetherian conditions.
\begin{enumerate}
\item The notion of weakly-Artinianity and quasi-Noetherianity are can be interpreted the stabilisation of a given chain of submodules.
    Indeed, let $\eta$ be a chain of torsion classes, $X$ be an object in $\A$.
    Then it follows from the definitions of torsion classes that if $r < s$ then $t_s X$ is naturally a subobject of $t_r X$ and $t_s t_r X$ is isomorphic to $t_s X$.
    Now, let $(a,b) \subset [0,1]$. 
    Then we can build a chain of subobjects of $X$ as follows.
    \begin{equation}\label{eq:chainoftorsion}\tag{\dag}
        0 \hookrightarrow t_b X \hookrightarrow \dots \hookrightarrow t_s X \hookrightarrow \dots \hookrightarrow t_r X \hookrightarrow \dots \hookrightarrow t_a X \hookrightarrow X 
    \end{equation}
    Hence $\eta$ is weakly-Artinian if the chain (\ref{eq:chainoftorsion}) stabilises to the left for every object in $X$ and every interval $(a,b)\subset [0,1]$.
    Dually, $\eta$ is quasi-Noetherian if the chain (\ref{eq:chainoftorsion}) stabilises to the right for every object in $X$ and every interval $(a,b)\subset [0,1]$.

    \item  Let $\A$ be a finite length category, for example $\A = \mod \Lambda$ for a finite dimensional algebra $\Lambda$. Then every chain of subobjects for an object in $\A$ is finite and therefore every chain of torsion classes is weakly-Artinian and quasi-Noetherian. 
    \item Let $\A$ be an abelian category with arbitrary coproducts, for example $\Mod \Lambda$. Then a chain of torsion classes $\eta$ in $\A$ is quasi-Noetherian and weakly-Artinian if and only if the set $ S = \{ i \in [0,1] \mid \mathcal{P}^\eta_i \neq \{0\}  \}$ is finite. Indeed, if this set is infinite,  then at each $i \in S$ take a non-zero $X_i \in \cP^\eta_i$ then $\eta$ does not  satisfy the weakly-Artinian nor quasi-Noetherian conditions at the coproduct $\coprod_{i \in S} X_i$.
    \end{enumerate}
\end{remark}

Let us give an example of an infinite chain of torsion classes in an abelian category of non-finite length that is weakly-Artinian and quasi-Noetherian.

\begin{example} \label{example:interesting chain}

    Let $k$ be a field and consider the category of functors from the poset $(\mathbb{R}_{\geq 0}, \leq)$ to $\operatorname{vect}_k$, the category of finite dimensional $k$-vector spaces. 
    Let $\A$ be the full subcategory of finitely presented functors. It is well known \cite{CB2015} that this is an abelian category with indecomposables corresponding to intervals $[x, y)$ where $x < y \in \mathbb{R}_{\geq 0} \cup \{ \infty\}$. Explicitly, the functor given by this interval is as follows
    \begin{align*}
        F_{[x,y)} :  \mathbb{R}_{\geq 0} & \to \operatorname{vect}_k 
        \\ a & \to \left\{ \begin{aligned} k & \quad a \in [x,y) \\ 0 & \quad \text{else.}
        \end{aligned} \right.
    \end{align*}
    A full set of isomorphism classes of indecomposable projectives is given by $\{P_x : = F_{[x,\infty)} \mid x \in \mathbb{R}_{\geq 0} \} $.
    Observe that there are no simple objects in $\A$.
    Hence $\A$ is not of finite length.
    
    Now, let $\beta : \mathbb{R}_{> 0}  \to (0,1) $ be an order preserving isomorphism and consider a countable divergent sequence $0= y_0 < y_1 < \dots $ in $\mathbb{R}_{\geq 0}$.
    We now define a chain of torsion classes $\eta = (\T_i)_{i \in [0,1]}$ as follows. 
    For $i \in (0,1)$ we define $\T_i$ to be the smallest torsion class containing $\{P_{y_j} \mid i \leq \beta(y_j) \}$ and, as always, $\T_0= \A$ and $\T_1 = \{0\}$. 
    Observe that the set $ S = \{ i \in [0,1] \mid \mathcal{P}^\eta(i) \neq \{0\}  \}$ is infinite.
    
    We claim that $\eta$ is weakly-Artinian and quasi-Noetherian. It is enough to verify this for the indecomposable objects of $\A$. To this end, let $F = F_{[x,y)} \in \A$ and $(a,b) \subset [0,1]$. 
    We consider the set $Z = \{ y_j \mid y_j \in [x,y) \text{ and } \beta(y_j) \in (a,b) \} $. Then if $Z$ is non-empty, it is finite. 
    Hence the chain of subobjects
    \[ F \supseteq t_a F \supseteq \dots \supseteq t_b F \] stabilises above at 
    $\beta(\inf Z)$ and below at $\beta (\sup Z)$. 
    If $Z$ is empty then we have that either $F = t_x F$ or $t_x F = 0$ for every $x\in(a,b)$ and there is nothing to show. 
    Thus, $\eta$ is weakly-Artinian and quasi-Noetherian. 
\end{example}

We note that the above example hints interesting connections with other mathematical disciplines.
Indeed, the category considered above has been deeply studied in topological data analysis \cite{OudotBook} and is reminiscent of the category of representations of the real line \cite{igusa2022continuous}.

\begin{lemma}\cite[Propositions 2.7 \& 2.9]{T-HN-filt} \label{lem:union-intersection-torsion}
Let $\eta = (\T_i)_{i\in[0,1]}$ be a chain of torsion classes in $\A$ and $(a,b) \subset [0,1]$ be an interval. Then
\begin{enumerate}
    \item If $\eta$ is quasi-Noetherian, then $(\bigcup_{r \in (a,b)} \T_r, \bigcap_{r \in (a,b)} \F_r$ ) is a torsion pair in $\A$. 
    \item If $\eta$ is weakly-Artinian, then $(\bigcap_{r \in (a,b)} \T_r, \bigcup_{r \in (a,b)} \F_r$ ) is a torsion pair in $\A$.
\end{enumerate}
\end{lemma}

\begin{remark}
We make some observations in the situation of the above lemma.
\begin{enumerate}
    \item We have that $\bigcup_{ r \in (a,b)} \T_r = \bigcup_{ r> a} \T_r$ and $\bigcap_{r \in (a,b)} \T_r = \bigcap_{r <b} \T_r$. Similar equalities hold for the torsionfree classes. 
    \item The intersection of torsion(free) classes is again a torsion(free) class. Thus the above lemma  tells us that the stated unions in 1. and 2. are torsion (resp. torsionfree) classes. To show this, the fact that $\eta$ is quasi-Noetherian is used to show that the union in 1. is closed under coproducts.
\end{enumerate}
\end{remark}

Associated to every indexed chain of torsion classes $\eta \in \CT(\A)$ there is a set 
$$\P^\eta = \{ \P^\eta_t \mid t \in [0,1]\}$$
of full subcategories of $\A$ where each $\P^{\eta}_t$ is defined as follows.

\begin{definition}\cite[Definition 2.12]{T-HN-filt}\label{def:chain}
Consider a chain of torsion classes $\eta\in \CT(\A)$.
Then, for every $t \in [0,1]$ we define the category $\P_t$ of \textit{$\eta$-quasisemistable} objects of \textit{phase} $t$ as:
$$
\P^\eta_t= \begin{cases}
\bigcap\limits_{s>0} \F_s  & \text{ if $t = 0$}\\
\left(\bigcap\limits_{s<t} \T_s\right)  \cap \left(\bigcap\limits_{s>t} \F_s\right) & \text{ if $t \in (0,1)$}\\
\bigcap\limits_{s<1} \T_s  & \text{ if $t = 1$}
\end{cases}
$$
\end{definition}

\begin{remark} Observe that $\bigcup_{r >a} \T_r \subseteq  \bigcap_{r <a} \T_r$. 
    In the case that $\eta$ is weakly-Artinian and quasi-Noetherian, in light of Lemma~\ref{lem:union-intersection-torsion}, the subcategory $\cP_t$ is then the \emph{heart of the twin torsion pair} given by the above torsion classes. It follows from \cite[Theorem 3.2]{tattar2021torsion} that $\cP_t$ is a quasi-abelian category.
\end{remark}

One of the main results in \cite{T-HN-filt} shows that every indexed chain of torsion classes $\eta \in \fT(\A)$ induces a Harder-Narasimhan filtration for every non-zero object $M \in \A$.
The formal statement is the following.

\begin{theorem}\cite[Theorem 2.13]{T-HN-filt}\label{thm:HN-filt}
Let $\A$ be an abelian category and $\eta \in \CT(\A)$.
Then every object $M\in \A$ admits a Harder-Narasimhan filtration with respect to $\eta$.
That is a filtration 
$$M_0 \subset M_1 \subset \dots \subset M_n$$
such that:
\begin{enumerate}
\item $0 = M_0$ and $M_n=M$;
\item there exists $r_k \in [0,1]$ such that $M_k/M_{k-1} \in \P_{r_k}$ for all $1\leq k \leq n$;
\item $r_1 > r_2 > \dots > r_n$.
\end{enumerate}
Moreover this filtration is unique up to isomorphism.
\end{theorem}

The notion of a slicing was introduced in \cite{Bridgeland2007} as an axiomatisation Harder-Narasimhan filtrations for triangulated categories. Here we introduce the corresponding concept for abelian categories. We also note that similar ideas have been studied in \cite{GKR2004} and also for exact categories \cite{Chen2010}.

\begin{definition}\label{def:slicing}
A \textit{slicing} $\P$ of the abelian category $\A$ consists of full additive subcategories $\P_r \subset \A$ for each $r \in [0,1]$ satisfying the following axioms:
\begin{enumerate}
\item If $M \in \P_r$ and $N \in \P_s$ with $r > s$, then $\Hom_{\A}(M, N)=0$;

\item For each nonzero object $M$ of $\A$ there exists a filtration 
$$0=M_0 \subset M_1 \subset \dots \subset M_n=M$$ 
of $M$ and a set 
$$\{ r_1, r_2, \dots, r_n \mid r_i > r_j  \text{ if $i < j$} \}\subset [0,1]$$
such that $M_i/ M_{i-1} \in \P_{r_i}$ which is unique up to isomorphism.
\end{enumerate}
We denote by $\Slice (\A)$ the set of all the slicings on an abelian category $\A$.
\end{definition}

It follows easily from Theorem~\ref{thm:HN-filt} that given a chain of torsion classes $\eta \in \CT(\A)$ then the set $\P_\eta$ defined above is a slicing in the sense of the previous definition. 
It was shown in \cite{T-HN-filt} that every slicing arises that way. 

\begin{theorem}\cite[Theorem~5.5]{T-HN-filt}\label{thm:slicingchains}
    Let $\eta \in \CT(\A)$ then $\cP_\eta$ is a slicing in $\A$. Moreover for every slicing $\P \in \Slice (\A)$ there is a chain of torsion classes $\eta \in \CT(\A)$ such that $\P = \P_\eta$.
\end{theorem}

\subsection{Weak stability conditions}

The notion of weak stability condition for an abelian category was introduced by Joyce in \cite{Joyce2007}, as a generalisation of the stability conditions introduced by Rudakov in \cite{Rudakov1997}.
In this subsection we recall the definition and prove/recall some of their properties.

\begin{definition}\cite[Definition 4.1]{Joyce2007}
A \textit{weak stability condition} on $\A$ is a function $\phi :  \Obj^*(\mathcal{A}) \to [0,1]$ from the non-zero objects, $\Obj^*(\A)$, of $\A$ to $[0,1]$, which is constant on isomorphism classes, such that for each short exact sequence  $0 \to L \to M \to N \to 0$ of non-zero objects in $\mathcal{A}$ we have that either 
$$\phi(L) \leq \phi(M) \leq \phi(N) \quad\text{ or } \quad\phi(L) \geq \phi(M) \geq \phi(N).$$

A weak stability condition, $\phi$, is a \emph{stability condition} in the sense of Rudakov \cite{Rudakov1997} if for all short exact sequences as above, we have that the so called see-saw condition:
$$\phi(L) < \phi(M) < \phi(N) \quad\text{ or } \quad \phi(L) = \phi(M) = \phi(N) \quad\text{ or } \quad\phi(L) > \phi(M) > \phi(N) $$ is satisfied.

A weak stability condition (on $\A$), $\phi$, is \textit{quasi-Noetherian} if for all chains of proper subobjects $A_0 \subset A_1 \subset \dots \subset X$ in $\A$, there exists an $n>0$ such that $\phi(A_n) > \phi(A_{n+1}/ A_n)$. Similarly, $\phi$ is \textit{weakly-Artinian} if for all chains of proper subobjects $X = A_0 \supset A_1 \supset \dots $ in $\A$, there exists an $n>0$ such that $\phi(A_{n+1}) < \phi(A_n / A_{n+1})$. 

By $\WStab \A$ we denote the class of all quasi-Noetherian and weakly-Artinian conditions on $\A$ with codomain $[0,1]$.
\end{definition}

\begin{remark} We make some remarks. 
\begin{enumerate}
    \item In the definition of weak stability condition we chose the closed interval $[0,1]$ as the codomain of the map $\phi$.
However, the definition of weak stability condition works for every totally ordered set. 
    \item We note that Joyce's original definition ask the map $\phi: \Obj^*(\A) \to [0,1]$ to factor through a quotient of the Grothendieck group $K_0(\A)$. 
In this note we drop this assumption from the definition.
   
    \item If $\A$ is a finite length category, then every weak stability condition on $\A$ is weakly-Artinian and quasi-Noetherian
\end{enumerate}
\end{remark}

\begin{definition}
Let $\phi$ be a weak stability condition on $\mathcal{A}$.
A non-zero object $M$ in $\mathcal{A}$ is said to be \textit{$\phi$-stable} (or \textit{$\phi$-semistable}) if every nontrivial subobject $L\subset M$ satisfies $\phi(L) < \phi(M / L)$ ($\phi(L)\leq \phi(M / L)$, respectively). 
\end{definition}

The following fact will be useful.

\begin{lemma} \label{lem:max-destab-subobj}
Let $\phi \in \WStab \A$. Then for all $0 \neq M \in \A$ there exists a subobject $M'$ of $M$ such that $M'$ is $\phi$-semistable and $\phi(M') \geq \phi (L)$ for all subobjects $L$ of $M$. Dually, there exists a quotient $M''$ of $M$ that is $\phi$-semistable such that $\phi(M'') \leq \phi (N)$ for all quotients $N$ of $M$. 
\end{lemma}
\begin{proof}
Let $M \in \A$ and let $L$ be a subobject of $M$. By \cite[Theorem 4.4, Step 4]{Joyce2007} there exists a $\phi$-semistable subobject $M'$ of $M$ (resp. $L'$ of $L$) such that $\phi(M') \leq \phi(S)$  (resp. $\phi(L') \leq \phi(S)$)  for every $\phi$-semistable subobject $S$ of $M$ (resp. of $L$). 
Moreover, by \cite[Theorem 4.4, Step 1]{Joyce2007} we have that $\phi(L') \geq \phi(L)$.  Note that by the transitivity of subobjects, $L'$ is a $\phi$-semistable subobject of $M$, thus  \[ \phi(L) \leq \phi(L') \leq \phi(M')\] and we are done. 
\end{proof}

\begin{remark}
In \cite[Section 4]{Joyce2007}, Joyce works with the assumptions that a weak stability condition is weakly-Artinian and that the ambient abelian category is Noetherian. We note that our replacement of the second condition by quasi-Noetherianity of the weak stability condition does not affect the validity of the results in \cite[Section 4]{Joyce2007}.
\end{remark}

Like for stability conditions, we will see that weak stability conditions induce chain of torsion classes in $\A$. 
This fact is a direct consequence of the next proposition. 
The proof of this proposition is a streamlined version of \cite[Proposition 2.14]{BSTpath}.

\begin{proposition}\label{prop:stabtorsion}
Let $\phi \in \WStab \A$.
For $p \in [0,1]$ consider the following subcategories of $\A$:
\begin{itemize}
\item $\T_{\geq p} := \{ M \in \A \mid \phi(N) \geq p \text{ for all quotients $N$ of $M$} \} \cup \{ 0 \}$; 
\item $\T_{>p} := \{ M \in \A \mid \phi(N) > p \text{ for all quotients $N$ of $M$} \} \cup \{ 0 \}$;
\item $\F_{\leq p} := \{ M \in \A \mid \phi(L) \leq p \text{ for all subobjects $L$ of $M$} \} \cup \{ 0 \}$ and; 
\item $\F_{< p} := \{ M \in \A \mid \phi(L) < p \text{ for all subobjects $L$ of $M$} \} \cup \{ 0 \}$. 
\end{itemize}
Then $(\T_{\geq p}, \F_{<p})$ and $(\T_{> p}, \F_{\leq p})$ are torsion pairs in $\A$.
\end{proposition}

\begin{proof}
We only prove that $(\T_{\geq p}, \F_{< p})$ is a torsion pair in $\A$ using the charactersisation of Theorem \ref{thm:dicksontorsion}(4), whence a similar proof also works for $(\T_{>p}, \F_{\leq p})$.
Let $f: M \to N$, where $M \in \T_{\geq p}$ and $ N \in \F_{< p}$.
Then we have that $\im f$ is a quotient of $M$ and a subobject of $N$. 
If $\im f$ is non-zero,  we have a contradiction as $\phi$ is a weak stability function. 
That is $\Hom_{\A} (M , N)=0$.

Now, consider $X$ such that $\Hom_\A (X,Y)=0$ for every $Y \in \F_{<p}$.
By Lemma~\ref{lem:max-destab-subobj} there exists a quotient $M'$ of $X$ such that $M'$ is $\phi$-semistable and $\phi(M')\leq \phi(N)$ for every quotient $N$ of $X$.
Since $\Hom_\A (X,Y)=0$ for every $Y \in \F_{<p}$, we conclude that $p \leq \phi(M')$.
Hence $\phi(N) \geq p$ for every quotient $N$ of $X$.
Then $X \in \T_{\geq p}$ by definition. 
One shows that every $Y$ such that $\Hom_\A (\T_{\geq p},Y)=0$ belongs to $\F_{> p}$ similarly.
\end{proof}

\begin{corollary}\label{prop:stab-chain}
Let $\phi \in \WStab \A$.
Then $\phi$ induces two chains of torsion classes $\eta^{-}_{\phi}, \eta^{+}_{\phi} \in \mathfrak{T}(\A)$ defined as follows
\begin{align*}
\eta^{+}_{\phi}:= & \{\T_0= \A, \T_1=\{0\} \text{ and }\T_s=\T_{\geq s} \mid s\in (0,1)\} \text{ and}
 \\ \eta^{-}_{\phi}:=& \{\T_0= \A, \T_1=\{0\} \text{ and }\T_s=\T_{> s} \mid s\in (0,1)\}.   
\end{align*}
\end{corollary}

\begin{proof}
We prove the statement for $\eta^{+}_{\phi}$, whence the argument is similar for $\eta^{-}_{\phi}$.

It was shown in Proposition \ref{prop:stabtorsion} that $\T_{\geq s }$ is a torsion class for every $s \in [0,1]$.
Suppose that $r \leq s$ and let $M \in \T_s$.
Then $\phi(N) \geq s$ for all quotients $N$ of $M$ by definition of $\T_{\geq s}$.
In particular $\phi(N) \geq s \geq r$ for all quotient $N$ of $M$.
Therefore $M \in \T_{\geq r}$.

It remains to show that $\eta_\phi^{+}$ is weakly-Artinian and quasi-Noetherian, we will only show the first of these, whence the other follows from a similar argument.
To this end, let $M \in \Obj^* (\A)$ and let $(a,b) \subset [0,1]$ be an interval. 
If $\eta^{+}_\phi$ is not weakly-Artinian, then for all $r \in (a,b)$ there exists $s \in (r,b)$ such that $t_r M$ is a proper subobject of $t_s M$. 
Using this, we construct an infinite chain of proper subobjects of $M$ indexed by a subset $I$ of $(a,b)$:
$$ M \supset t_a M \supset \dots \supset t_r M \supset t_s M \supset \dots  \supset t_b M. $$
By the weakly-Artinian property of $\phi$, there exists $u,v  \in I$ with $u < v$ such that $\phi (t_uM) < \phi( t_v M / t_u M) = \phi (f_u t_v M)$. However, as $t_u M \in \T_{\geq  u}$ and $f_u t_v M \in \F_{< u}$ we have that $\phi(t_uM) \geq u$ and $\phi (f_u t_v M) < u$ which is a contradiction. 
\end{proof}

The key result that allow us to compare chains of torsion classes and weak stability condition is the following.

\begin{theorem}\label{thm:semistabchain}
Let $\phi \in \WStab \A$, $\eta^{+}_{\phi}, \eta^{-}_{\phi}\in \CT(\A)$ as above, and let $t \in [0,1]$.
Then 
$$\P^{+}_t= \P^{-}_t = \{M \in \A \mid M \text{ is $\phi$-semistable and } \phi(M)=t\} \cup \{0\},$$
where $\P^{+}_\phi$ and $\P^{-}_\phi$ are the slicings of $\eta^{+}_{\phi}$ and $\eta^{-}_{\phi}$, respectively.
\end{theorem}

\begin{proof}
We will show the statement for $\eta^+_\phi$. 
A similar argument holds for $\eta^-_\phi$.

By definition, we have that $\P^+_t=\bigcap\limits_{s<t} \T_{\geq s}  \cap \bigcap\limits_{s>t} \F_{< s}$.
It follows that $\bigcap\limits_{s<t} \T_s= \T_{\geq t}$ and $\bigcap\limits_{s>t} \F_s = \F_{\leq t}$.
Is clear that $\phi(M)=t$ for all $M \in \bigcap\limits_{s<t} \T_{\geq s}  \cap \bigcap\limits_{s>t} \F_{< s}$.
Moreover, for every subobject $L$ of $M$, $\phi(L) \leq t$ because $M \in \F_{\leq t}$ and, similarly, $\phi(M/L) \geq t$ because $M \in \T_{\geq t}$. 
Therefore $M$ is a $\phi$-semistable object of phase $t$.

In the other direction, suppose that $M$ is a $\phi$-semistable object of phase $t$.
Then we have that $\phi(L) \leq \phi(M)=t$ for every subobject $L$ of $M$ because $M$ is $\phi$-semistable, implying that $M \in \F_{\leq t}$. 
Dually, $t = \phi(M) \leq \phi(N)$ because $M$ is $\phi$-semistable, so $M \in \T_{\geq t}$.
Therefore $M \in \T_{\geq t} \cap \F_{\leq t} = \P^+_t$, as claimed.
\end{proof}

\bigskip

In the situation of the above Theorem, we define $\cP^+_\phi (= \cP^-_\phi)$ to be the \emph{slicing of} $\phi$. As a consequence of Theorem \ref{thm:HN-filt} and Theorem \ref{thm:semistabchain} we recover several important results on weak stability conditions.

\begin{corollary}\cite[Theorem 4.4]{Joyce2007}\label{cor:stabHNfilt}
Let $\phi \in \WStab \A$. 
Then every object $M\in \A$ admits a Harder-Narasimhan filtration.
That is a filtration 
$$M_0 \subset M_1 \subset \dots \subset M_n$$
such that:
\begin{enumerate}
\item $0 = M_0$ and $M_n=M$;
\item $M_k/M_{k-1}$ is $\phi$-semistable;
\item $\phi(M_1) > \phi(M_2/M_{1}) > \dots > \phi(M_n/M_{n-1})$.
\end{enumerate}
Moreover this filtration is unique up to isomorphism.
\end{corollary}

\begin{proof}
Let $\phi \in \WStab \A$.
Then Proposition \ref{prop:stab-chain} says that $\phi$ induces two chains of torsion classes $\eta^{+}_{\phi}, \eta^{-}_{\phi}\in \CT(\A)$.
Then Theorem \ref{thm:HN-filt} implies that for every non-zero object $M$ of $\A$ there are two filtrations 
$$M_0 \subset M_1 \subset \dots \subset M_n$$
of $M$, one associated to $\eta^{+}_{\phi}$ and the other to $\eta^{+}_{\phi}$, such that:
\begin{enumerate}
\item $0 = M_0$ and $M_n=M$;
\item $M_k/M_{k-1} \in \P_{r_k}$ for some $r_k \in [0,1]$ for all $1\leq k \leq n$;
\item $r_1 > r_2 > \dots > r_n$;
\end{enumerate}
which are unique up to isomorphism. 
It follows from Theorem \ref{thm:semistabchain} that both filtrations coincide since $\eta^{+}_{\phi}$ and $\eta^{-}_{\phi}$ induce the same slicing.
Moreover, Theorem~\ref{thm:semistabchain} also implies that $M_k/M_{k-1}$ is $\phi$-semistable for each $k$ and that $\phi(M_1) > \phi(M_2/M_{1}) > \dots > \phi(M_n/M_{n-1})$.
This finishes the proof.
\end{proof}

\begin{remark}
The quasi-Noetherian and weakly-Artinian conditions are key in Theorem \ref{thm:HN-filt} and Corollary \ref{cor:stabHNfilt} in giving the finitness of the Harder-Narasimhan filtrations in each direction. 
In the triangulated setting of Bridgeland, a stability condition (in the triangulated sense) consists of a bounded t-structure and a weakly-Artinian, quasi-Noetherian stability condition (in the abelian sense) on its heart \cite[Proposition 5.3]{Bridgeland2007}. Here, the boundedness of the t-structure also plays a key role in ensuring the finiteness of the Harder-Narasimhan filtrations. 
\end{remark}

The above result also gives us alternative description of the torsion classes induced by weak stability conditions. 

\begin{corollary} \label{cor:torsionclasses-other-description}
Let  $\phi \in \WStab \A$. Then 
\[ \T_{\geq p} := \Filt \{ M \in \A \mid \phi(M) \geq p \text{ and } M \text{ is } \phi \textrm{-semistable} \} \cup \{0\}. \]

\end{corollary}
\begin{proof}
($\supseteq$)  Since torsion classes are closed under extensions, it suffices to verify that each $\phi$-semistable object $M$ with $\phi(M) \geq p$ is in $\T_{\geq p}$. To this end, let $M$ be such an object and let $0 \to L \to M \to N \to 0$ be a short exact sequence in $\A$. If $\phi(N) < p$ we must have that $\phi(L) \geq \phi(M) > \phi(N)$ as $\phi$ is a weak stability condition. However, this is a contradiction to the fact that $M$ is $\phi$-semistable. 

($\subseteq$) Let $0\neq M \in \T_{\geq p}$. By Corollary~\ref{cor:stabHNfilt} there exists a unique sequence $r_1 > r_2 > \dots > r_n$ in $[0,1]$ and a filtration 
$0 = M_0 \subset M_1 \subset \dots \subset M_n= M$ of $M$ such that $M_k / M_{k-1}$ is $\phi$-semistable of phase $r_k$.
As $M \in \T_{\geq p}$, $r_n = \sup\{ i \in [0,1] \mid M \in \mathcal{T}_i \}  \geq p$. Thus $M$ admits a filtration by $\phi$-semistable objects with phase at least $p$. 
\end{proof}

\section{From chains of torsion classes to weak stability conditions} \label{sec:tors-to-wsc}

In the previous section we have shown how one can go from weak stability conditions to chains of torsion classes. 
In this section we show a converse construction and characterise when this construction gives honest stability conditions. 
Moreover we see in  that there are natural equivalence relations $\sim$ on $\CT(\A)$ and $\WStab \A$ induced by the slicings such that our constructions induce  bijections between $\CT(\A)/\sim$ and $\WStab \A /\sim$.

\begin{definition} \label{def:ohm-mho}
Let $\eta = (\mathcal{T}_i)_{i \in [0,1]} $ be a chain of torsion classes in $\A$.
We introduce the maps
\begin{align*}
    \mho_\eta = \mho : \Obj^\ast \A & \longrightarrow [0,1] \\ M & \longmapsto \sup \{ i \in [0,1] \mid M \in \mathcal{T}_i \}
\end{align*}
and 
\begin{align*}
    \Omega_\eta = \Omega : \Obj^\ast \A & \longrightarrow [0,1] \\ M & \longmapsto  \inf\{ i \in [0,1] \mid M \in \mathcal{F}_i \}
\end{align*}
where $\mathcal{F}_i = \mathcal{T}_i^{\perp}.$
\end{definition}

\begin{remark} 
Since the interval $[0,1]$ is complete,  the equalities $\Omega (M) = \sup\{ i \in [0,1] \mid M \not\in \mathcal{F}_i \}$    and $\mho (M) = \inf \{ i \in [0,1] \mid M \not\in \mathcal{T}_i \}$ hold.
\end{remark}

\begin{remark}
If $\eta$ is quasi-Noetherian and weakly-Artinian, then $\Omega_{\eta}$ (respectively, $\mho_{\eta}$) returns $r_1$ (respectively $r_n$) from the Harder-Narasimhan filtration of $M$ (Theorem~\ref{thm:HN-filt}).
We note that the definitions of $\Omega$ and $\mho$ also make sense for chains of torsion classes that do not satisfy these extra conditions.
\end{remark}

We make some observations.

\begin{lemma}\label{lemma: omega properties} Let $\eta$ be a chain of torsion classes in $\A$ and let $0 \to L \to M \to N \to 0$ be a short exact sequence in $\mathcal{A}$. Then
\begin{enumerate}
    \item $\mho (M) \leq \mho(N)$.
    \item $\mho (M) \geq \min \{ \mho (L), \mho (N) \}$.
\end{enumerate}
Dually,
\begin{enumerate}
    \item[1'.] $\Omega (M) \geq \Omega (L)$.
    \item[2'.] $\Omega (M) \leq \max\{ \Omega (L), \Omega (N) \}$.
\end{enumerate}
\end{lemma}
\begin{proof} We only prove the $\mho$ statements, the others follow by dual arguments. 
The first assertion follows from the fact that torsion classes are closed under quotients. The second follows since torsion classes are closed under extensions. 
\end{proof}

\begin{theorem} \label{thm:ohm-mho-are-wsc} Let $\eta$ be a chain of torsion classes on $\A$. 
$\Omega : \Obj^\ast \A \to [0,1]$ and $\mho : \Obj^\ast \A \to [0,1]$ are both weak stability conditions on $\mathcal{A}$.
\end{theorem}
\begin{proof}
Let $0 \to L \to M \to N \to 0$ be a short exact sequence in $\mathcal{A}$. First, suppose that $\mho (L) \leq \mho (N)$ then, by Lemma \ref{lemma: omega properties}, $\mho (M) \geq \min \{ \mho (L), \mho (N) \}= \mho (L)$ and  $\mho (M) \leq \mho(N)$. Together, we have that $\mho (L) \leq \mho (M) \leq \mho (N)$. 

Now suppose that $\mho (N) \leq \mho (L)$. Then, by Lemma \ref{lemma: omega properties}, $\mho (M) \geq \min \{ \mho (L), \mho (N) \}= \mho (N)$ and $\mho (M) \leq \mho (N))$. Thus, $\mho (M) = \mho (N)$  and we have that $\mho (L) \geq  \mho (M) = \mho (N)$. 

The fact that $\Omega$ is a weak stability condition follows from a dual argument.
\end{proof}

\begin{remark}
For a chain of torsion classes $\eta$ that is weakly-Artinian and quasi-Noetherian, the weak stability conditions $\Omega_\eta$ and $\mho_\eta$ are not necessarilty weakly-Artinian nor quasi-Noetherian. For example, let $\A = \Mod \Lambda$ and let $\eta$ be the chain  $0 \subset \A$. 
Thus the notion of weakly-Artinian and quasi-Noetherian chains of torsion classes is more general than the correspondoning notion for weak stability conditions.
This can actually be seen from the definitions.
Indeed weakly-Artinianity and quasi-Noetherianity for weak stability conditions asks for a certain property to be satisfied in every chain of subobjects, whilst weakly-Artinianity and quasi-Noetherianity for chains of torsion classes the property needs to be satisfied in only one chain of subobjects, see Remark~\ref{rmk:WA&QN}.1. 
\end{remark}

It has been shown in Theorem~\ref{thm:semistabchain} that every weak stability condition induces chains of torsion classes.
In particular, every stability condition induces chains of torsion classes, see also \cite{T-HN-filt}.
However, in practice, $\Omega$ and $\mho$ are rarely stability conditions.

\begin{proposition} \label{prop:when-ohm-mho-stab}
Let $\eta$ be a chain of torsion classes on $\A$. 
Then the following are equivalent 
\begin{enumerate}
    \item There exists a unique $t \in [0,1]$ such that $\cP^\eta_t \neq \{0\}$ (and hence $\cP^\eta_t = \A$).
    \item $\mho$ is constant.
    \item[2'.] $\Omega$ is constant.
    \item $\mho$ is a stability condition on $\mathcal{A}$.
    \item[3'.] $\Omega$ is a stability condition on $\mathcal{A}$.
    \item $\mho = \Omega$.
\end{enumerate}
\end{proposition}
\begin{proof}
Clearly 4 $\Leftarrow$ 1 $\Leftrightarrow$ 2 $\Rightarrow $ 3. To show 3 $\Rightarrow$ 2, suppose that $\mho$ is not constant and let $A, B \in \Obj^{\ast}\mathcal{A}$ be semistable such that $\mho (A) \not = \mho(B) $. Without loss of generality, we may suppose that  $\mho  (A)  > \mho (B) $.
Then $\mho ( A \oplus B) = \mho (B)$ and  the split short exact sequence $0 \to A \to A \oplus B \to B \to 0$ in $\mathcal{A}$ does not satisfy the see-saw condition, thus $\mho$ is not a stability condition. 

Finally, to see 4 $\Rightarrow$ 2, observe that if $\Omega = \mho$ then all Harder-Narasimhan filtrations are of length one, that is, all objects are semistable. If $\mho$ is not constant, say $\mho(A) \not = \mho (B)$, then $A \oplus B$ is not semistable, contradicting the above. The remaining statements for $\Omega$ are shown in a similar way.
\end{proof}

\subsection{Order structures in chains of torsion classes} \label{sec: order structures}

We now investigate how the poset structure of $(\Tors \A, \subseteq)$ induce partial orders on both $\fT (\A)$ and $\WStab(\A)$. 
\begin{definition}
Let $\eta = (\T_i)_{i \in [0,1]}, \eta ' = (\T_i')_{i \in [0,1]} \in \fT (\A)$ then 
\begin{align*}
     \eta \leq \eta' &: \Leftrightarrow \T_i \subseteq \T_i' \quad  \forall i \in [0,1].
     \intertext{Let $\phi, \phi ' \in \WStab(\A)$ then }
     \phi \leq \phi' &: \Leftrightarrow \phi (M)   \leq  \phi '(M) \quad \forall M \in \A.
\end{align*}
\end{definition}

\begin{proposition}
With the notation above $(\CT(\A), \leq)$ and $(\WStab \A, \leq)$ posets.
\end{proposition}

\begin{remark} \label{rem:mappings order preserving}
We make some observations
\begin{enumerate}
    \item $\mho , \Omega: \fT(\A) \rightrightarrows \WStab(\A)$ and $\eta^+,\eta^- :  \WStab(\A) \rightrightarrows \fT(\A) $ are all order preserving mappings. 
    \item For all $\eta \in \fT(\A)$, $\mho_\eta \leq \Omega_\eta$.
    \item For all $\phi \in \WStab(\A)$, $\eta^-_\phi \leq \eta^+_\phi$.
    \item These natural orders may also be defined for chains of torsion classes (resp. weak stability conditions) that are not necessarily weakly-Artinian or quasi-Noetherian. 
\end{enumerate}
\end{remark}

\begin{proposition} \label{prop:comparision of chains}
Let $\sigma = (\mathcal{S}_i)_{i\in[0,1]} \in \fT(\mathcal{A}).$ Then 
\[ \eta^- \mho_\sigma \leq  \eta^- \Omega_\sigma \leq \sigma \leq \eta^+ \mho_\sigma \leq \eta^+ \Omega_\sigma. \] 
\end{proposition}

\begin{proof}
By Proposition \ref{prop:stabtorsion} and Corollary \ref{prop:stab-chain}, $\eta^+ \mho_\sigma = (\T_{p})_{p\in[0,1]}$ where 
\[ \T_p = \{ M \in \A \mid \mho(N) \geq p \text{ for all quotients $N$ of $M$} \} \cup \{ 0 \}. \] By Lemma \ref{lemma: omega properties}, $\mho(N) \geq \mho(M)$ for all quotients $N$ of $M$, thus 
\begin{align*}
    \T_p &= \{ M \in \A \mid \mho(M) \geq p\} \cup \{ 0 \} 
    \\ & = \{ M \in \A \mid \sup \{ i \in [0,1] \mid M \in \mathcal{S}_i \} \geq p\} \cup \{ 0 \}.
\end{align*}  
The inclusion $\mathcal{S}_p \subseteq \T_p$ follows.
Similarly, we have that $\eta^- \Omega_\sigma = (T'_p)_{p\in [0,1]}$ where \[  (\T'_p)^{\perp} = \{ M \in \A \mid \inf\{ i \in [0,1] \mid M \in S_i^\perp \} \leq p\} \cup \{ 0 \}\] and the inclusion $\mathcal{S}_p^{\perp} \subseteq \T'_p $ follows.  The remaining inequalities follow from Remark~\ref{rem:mappings order preserving}. 
\end{proof}

\begin{proposition}
Let $\phi$ be a weak stability condition on $\A$. Then 
\[
    \mho \eta^-_\phi  \leq   \mho \eta^+_\phi  \leq \phi \leq \Omega \eta^-_\phi \leq \Omega \eta^+_\phi. 
\]
\end{proposition}
\begin{proof}
The first and last inequalities follow from Remark~\ref{rem:mappings order preserving}. It remains to show the others.
 Let $M$ be a non-zero object in $\A$, then 
$\mho \eta^+_\phi (M) = \sup I$ where  \[ I =\{i \in [0,1] \mid \phi (N) \geq i \text{ for all quotients $N$ of $M$} \}.
\]
$\phi (M)$ is clearly an upper bound of $I$, thus $\sup I \leq \phi (M)$. On the other hand, we have that  $\Omega \eta^-_\phi(M) = \inf J$ where 
\[ J  = \{ i \in [0,1] \mid \phi(L)  <i \text{ for all submodules $N$ of $M$} \}. \]
As $\phi(M)$ is a lower bound of $J$, we have $\Omega \eta^+_\phi(M) \geq \phi(M)$. 
\end{proof}

We now show that the slicings induced by the chains of torsion classes in Proposition~\ref{prop:comparision of chains} coincide. 

\begin{theorem} \label{prop:sameslicing} Let $\sigma = (\mathcal{S}_i)_{i\in[0,1]} \in \fT(\mathcal{A})$. Then the chains of torsion classes, $\eta^- \mho_\sigma$,  $\eta^- \Omega_\sigma,$ $\sigma$, $\eta^+ \mho_\sigma$, $\eta^+ \Omega_\sigma$ all have the same slicing. 
\end{theorem}
\begin{proof}
We show that $\eta^- \mho_\sigma$ has the same slicing as $\sigma$, whence the proof for the other chains of torsion classes will follow by similar arguments. 
It follows from the definitions that $\eta^- \mho_\sigma = (\mathcal{T}_p)_{p\in [0,1]}$ where 
\[ \mathcal{T}_p = \{ M \in \A \mid \sup\{ i \mid M \in \mathcal{S}_i\} >p \} \cup \{0\}, \] and has a torsionfree classes given by 
\[ \mathcal{F}_p = \mathcal{T}_p^{\perp} = \{ M \in \A \mid \sup\{ i \in [0,1] \mid L \in \cS_i \} \leq p \text{ for all submodules $L$ of $M$} \} \cup \{0\}. 
\] 
By Definition \ref{def:chain}, the slicing of $\sigma$  at $r \in [0,1]$ is given by 

\[ \cP_r^S  = \Big( \bigcap_{s<r} \mathcal{S}_s \Big) \cap \Big(  \bigcap_{s>r} (\mathcal{S}_s)^\perp \Big) \] and similarly for $\cP_r^{\eta^-\mho \sigma}$.
Thus, it is enough to show for all $r \in [0,1]$ that
\[ \bigcap_{s<r} \mathcal{S}_s = \bigcap_{s<r} \mathcal{T}_s  \text{ \quad (A)\quad and \quad} \bigcap_{s>r} (\mathcal{S}_s)^\perp = \bigcap_{s>r} \F_s  \text{\quad (B).} \]

(A): By Proposition~\ref{prop:comparision of chains}, we have that $\T_r \subseteq \cS_r$ for all $r \in [0,1]$ and the $(\supseteq)$ inclusion follows. 
To show the $(\subseteq)$ inclusion, let $M \in \cS_s$ for all $s<r$. Then $\sup\{i \in [0,1] \mid M \in\cS_i\} \geq r$ and so $M \in T_s$ for all $s <r$. 

(B): We first show the $(\subseteq)$ inclusion. Let $M \in \cS_s^\perp$ for all $ s > r$ and let $L$ be a submodule of $M$. Then, as torsionfree classes are closed under subobjects (Theorem~\ref{thm:dicksontorsion}(3)), we have that $L \in \cS_s^\perp$ for all $ s > r$. It follows that  $L \not\in \cS_s$ for all $s>r$ and thus $\sup\{ i \in [0,1] \mid L \in \cS_i \} \leq r$ and we conclude from the above description that $M \in \F_r$.

We now show the reverse inclusion. It follows from Theorem~\ref{thm:dicksontorsion}(3), that the intersection of torsionfree classes is again a torisonfree class, thus both subcategories that we are considering are torsionfree classes, that is, there are torsion pairs $(\mathcal{X}, \bigcap_{s>r} (\mathcal{S}_s)^\perp ) $ and $( \mathcal{Y},  \bigcap_{s>r} \F_s) $. We will show that $\mathcal{X} \subseteq \mathcal{Y}$ whence it follows that $\bigcap_{s>r} \F_s \subseteq \bigcap_{s>r} (\mathcal{S}_s)^\perp$.

As $\sigma$ is quasi-Noetherian, by Lemma~\ref{lem:union-intersection-torsion}, we have that $\mathcal{X} = \bigcup_{s>r}\cS_s$. Clearly, we also have that $\bigcup_{s>r}\T_s \subseteq \mathcal{Y}$. Thus it is enough to show that $\bigcup_{s>r}\cS_s \subseteq \bigcup_{s>r}\T_s$. Indeed, let let $M \in \cS_s$ for some $s>r$, then for $s>s'>r$, $M \in T_{s'}$ and we are done.
\end{proof}

\begin{corollary} \label{cor:max-min-slicing}
 Let $\sigma = (\mathcal{S}_i)_{i\in[0,1]} \in \fT(\mathcal{A})$. Then $\eta^+ \Omega_\sigma$ (resp. $\eta^- \mho_\sigma$) is  the maximal (resp. minimal) chain of torsion classes  with respect to having the same slicing as $\sigma$. That is, for all $S' \in \mathfrak{T}$ such that the slicing of $\sigma'$ and $\sigma$ coincide, then $\eta^+ \Omega_\sigma \leq \sigma' \leq \eta^- \mho_\sigma $. 
\end{corollary}
\begin{proof}
If the slicings of $\sigma$ and $\sigma'$ coincide then $\Omega_\sigma = \Omega_{\sigma'}$ and $\mho_\sigma = \mho_{\sigma'}$. Whence the claim follows from Proposition \ref{prop:comparision of chains}. 
\end{proof}

\subsection{Total stability}
There has been interest on studying the stability conditions that make every object in the category semistable (see \cite{DiazGilbertKinser} and the references therein). 
We now give a characterisation of the chains of torsion classes having this property. 
By $\WStab^{\tot}\A$ we denote the class of weak stability conditions $\phi$ on  $\A$ such that every indecomposable object $M\in \A$ is $\phi$-semistable.
Recall that a torsion pair $(\T, \F)$ in an abelian category $\A$ is called \emph{split} if $\Ext^1_\A (\F, \T) = 0 $ or, equivalently, that every indecomposable object in $\A$ belongs to $\T$ or $\F$.

\begin{proposition} \label{thm:split-tors-total-wstab}
Let $\phi \in \WStab^{\tot}\A$ for an abelian category $\A$. Then for all $p\in [0,1]$, $(\T_{\geq p}, \F_{<p})$ and $(\T_{> p}, \F_{\leq p})$ are both split torsion pairs. Conversely, let $\eta = (\T_i)_{[0,1]} \in \mathfrak{T}(\A)$ such that $(\T_i, \F_i)$ is split for all $i \in [0,1]$ then each indecomposable $M \in \A$ is $\eta$-quasisemistable and $\Omega_\eta, \mho_\eta \in \WStab^{\tot}\A$.   
\end{proposition}
\begin{proof}
Let $p \in [0,1]$, we prove the claim for the pair $(\T_{\geq p}, \F_{<p})$, the other case being similar. By Proposition~\ref{prop:stabtorsion}, the pair is a torsion pair and it remains to show that it is split. Let $M \in \A$ be indecomposable and let  $0 \to T \to M \to F \to 0$ be the canonical short exact sequence of $M$ with respect to $(\T_{\geq p}, \F_{<p})$. By Corollary~\ref{cor:torsionclasses-other-description} we have that $\T_{\geq p}$ (resp. $\F_{< p}$) contains all $\phi$-semistable objects of phase at least $p$ (resp. less than $p$). By assumption $M$ is $\phi$-semistable, thus if $\phi(M) \geq p$ then $M \cong T$ or else $M \cong F$. We conclude that  $(\T_{\geq p}, \F_{<p})$ is split. 

Now let $\eta = (\T_i)_{i \in [0,1]}$ and suppose that there exists $X \in \A$ indecomposable such that $X$ is not $\eta$-quasisemistable. At $p = \Omega_\eta M$, the canonical short exact sequence  \[ 0 \to t_p M \to M \to f_p \to 0 \]  of the torsion pair $(\T_p, \F_p)$ is non-split. 
The final claim follows from Proposition~\ref{prop:sameslicing} and Theorem~\ref{thm:semistabchain}. 
\end{proof}

\subsection{An equivalence relation on \texorpdfstring{$\CT(\A)$}{CT(A)} and WStab\texorpdfstring{$\A$}{A}} \label{section:equivrelation}

As we have seen before, the properties of chains of torsion classes and weak stability conditions rely heavily on the underlying slicing. 
This induces a natural equivalence relation as follows.

\begin{definition}
We define a natural equivalence relation $\sim$ on $\CT(\A)$ defined by $\eta \sim \eta'$ if and only if $\P^\eta_r = \P^{\eta'}_r$ for all $0 \leq r \leq 1$.

Similarly, we define an equivalence relation $\sim$ in $\WStab \A$ defined by $\phi \sim \phi'$ whenever every object $M\in \Obj^*(\A)$ is $\phi$-semistable if and only if it is $\phi'$-semistable and, if this is the case, then $\phi(M)=\phi'(M)$.
\end{definition}

As a direct consequence of the previous definition and Theorem~\ref{thm:semistabchain} we obtain the following fact.

\begin{proposition} \label{prop: slice(A)  cong T(A)}
The spaces $\WStab \A/\sim$ and $\CT(\A) / \sim$ coincide. Moreover both coincide with $\Slice (\A)$. 
\end{proposition} 
\begin{proof}
The fact that $\WStab \A/\sim$ and $\CT(\A) / \sim$ coincide follows directly from Theorem~\ref{thm:semistabchain}. The moreover part of the statement follows from Theorem~\ref{thm:slicingchains}.
\end{proof}

\begin{remark}
It follows from Proposition \ref{prop:sameslicing} and Corollary \ref{cor:max-min-slicing} that for a chain of torsion classes $\sigma \in \mathfrak{T} (\A)$, that  $\eta^- \mho_\sigma $,  $\eta^- \Omega_\sigma ,$ $\sigma $, $\eta^+ \mho_\sigma $, $\eta^+ \Omega_\sigma $ are all in the same equivalence class and that $\eta^+ \Omega_\sigma $ and $\eta^- \mho_\sigma $ can be thought of as minimal and maximal representatives of this equivalence class respectively.
\end{remark}


\section{The space of chains of torsion classes}\label{sec:spacetor}

In this section we follow ideas of \cite{Bridgeland2007} to show that $\CT(\A)$ has a natural topological structure and we study some of its properties.
Our main strategy consists of lifting topological properties from the closed interval $[0,1]$ of the real numbers with the canonical topology to $\CT(\A)$. 

In the next result we prove that the function $d: \CT(\A) \times \CT(\A) \to \mathbb{R}$ defined in \cite[Section 6]{Bridgeland2007} is a pseudometric in $\CT(\A)$; that is, the following conditions hold:
\begin{enumerate}
    \item $d(x,x)=0$ for all $x \in \CT(\A)$;
    \item $d(x,y)=d(y,x)$ for all $x, y \in \CT(\A)$;
    \item $d(x,y) \leq d(x,z) + d(z,y)$ for all $x,y,z \in \CT(\A)$.
\end{enumerate}
In other words a pseudometric is a metric where the distance between two different points might be $0$.
The reader should be aware that in the following statement we will be using the functions $\Omega_\eta$ and $\mho_\eta$  as defined in Definition~\ref{def:ohm-mho}.

\begin{proposition}\label{prop:metric}
Define the function $d: \CT(\A) \times \CT(\A) \to \mathbb{R}$ as 
$$d(\eta_1, \eta_2):= \sup \left\{ \left|\mho_{\eta_1}(M)-\mho_{\eta_2}(M)\right|, \left|\Omega_{\eta_1}(M)-\Omega_{\eta_2}(M)\right| \mid  M \in \Obj^{\ast} \A  \right\}$$
where $\eta_i \in \CT(\A)$.
Then the function $d: \CT(\A) \times \CT(\A) \to \mathbb{R}$ is well-defined and induces a pseudometric on the space $\CT(\A)$.
\end{proposition}

\begin{proof}
The fact that $d(\eta, \eta)=0$ and that $d(\eta_1, \eta_2) = d(\eta_2, \eta_1)$ are immediate from the definition. 
So we only show that $d(\eta_1, \eta_3) \leq d(\eta_1, \eta_2) +d(\eta_2, \eta_3)$.

In order to do that, let $M \in \Obj^{\ast} \A$. 
Then the following hold
\[
\left|\mho_{\eta_1}(M)-\mho_{\eta_3}(M)\right| \leq \left|\mho_{\eta_1}(M)-\mho_{\eta_2}(M)\right| +\left|\mho_{\eta_2}(M)-\mho_{\eta_3}(M)\right|,
\]
\[
\left|\Omega_{\eta_1}(M)-\Omega_{\eta_3}(M)\right| \leq \left|\Omega_{\eta_1}(M)-\Omega_{\eta_2}(M)\right| +\left|\Omega_{\eta_2}(M)-\Omega_{\eta_3}(M)\right|.
\]
Therefore we have the following series of inequalities.
\begin{align*}
    d(\eta_1, \eta_3)  & = \sup \left\{ \left|\mho_{\eta_1}(M)-\mho_{\eta_3}(M)\right|, \left|\Omega_{\eta_1}(M)-\Omega_{\eta_3}(M)\right|  \mid  M \in \Obj^{\ast} \A  \right\}\\
     & \leq \sup \left\{ \left|\mho_{\eta_1}(M)-\mho_{\eta_2}(M)\right| + \left|\mho_{\eta_2}(M)-\mho_{\eta_3}(M)\right|, \right.\\
    & \phantom{= \sup (( }  \left.  \left|\Omega_{\eta_1}(M)-\Omega_{\eta_2}(M)\right| + \left|\Omega_{\eta_2}(M)-\Omega_{\eta_3}(M)\right|   \mid  M \in \Obj^{\ast} \A   \right\} \\
    & \leq \sup \left\{ \left|\mho_{\eta_1}(M)-\mho_{\eta_2}(M)\right|, \left|\Omega_{\eta_1}(M)-\Omega_{\eta_2}(M)\right|  \mid  M \in \Obj^{\ast} \A   \right\} \\
     & \phantom{=} \, + \sup \left\{   \left|\mho_{\eta_2}(M)-\mho_{\eta_3}(M)\right| , \left|\Omega_{\eta_2}(M)-\Omega_{\eta_3}(M)\right|  \mid  M \in \Obj^{\ast} \A   \right\} \\
     & = d(\eta_1, \eta_2) +  d(\eta_2, \eta_3).
\end{align*}
\end{proof}

Even if the definition of the function $d$ is concise, given two different chain of torsion classes $\eta_1$ and $\eta_2$ their distance $d(\eta_1, \eta_2)$ might be rather difficult to compute.
The next result of \cite{Bridgeland2007} give us a different way to calculate this distance. 

\begin{lemma}\cite[Lemma 6.1]{Bridgeland2007}\label{lem:distance}
Let $\eta_1, \eta_2 \in \fT(\A)$.  
Then 
$$d(\eta_1, \eta_2)= \inf \left\{ \varepsilon \in \mathbb{R} \mid \P^{\eta_2}_r \subset \Filt\left(\bigcup_{t\in [r-\varepsilon, r+\varepsilon]}\P_t^{\eta_1} \right)\text{ for all $r \in [0,1]$}\right\}.$$
\end{lemma}

An important corollary of Proposition \ref{prop:metric} is that $\CT(\A)$ is a topological space and that we can give an explicit basis for that topology.

\begin{corollary}\label{cor:basis}
 $\CT(\A)$ is a topological space with a basis given by the sets 
$$B_\varepsilon(\eta):= \{ \eta' \in \CT(\A) \mid  d(\eta, \eta') < \varepsilon\}$$
for all $\eta \in \CT(\A)$ and $0 < \varepsilon < 1$.
\end{corollary}

The next result gives a sufficient condition for two chains to be a distance $\varepsilon < \frac{1}{2}$ apart.  

\begin{proposition}
Let $\eta$ be the chain of torsion classes in $\CT(\A)$ such that $\T_r=\X$ for all $r \in (0,1)$, where $\X$ is a fixed torsion class in $\A$. 
If $0 <\varepsilon < \frac{1}{2}$, then $\eta' \in B_{\varepsilon}(\eta)$ if and only if $\T'_r = \X$ for all $r \in [\varepsilon, 1 - \varepsilon]$.
\end{proposition}

\begin{proof}
As $\X$ is a torsion class in $\A$, there exist a torsionfree class $\Y$ of $\A$ such that $(\X, \Y)$ is a torsion pair in $\A$.
Then the slicing of $\eta$ is $\P_0=\Y$, $\P_1=\X$ and $\P_r=\{0\}$ if $r \in (0,1)$.

Now consider $\eta'\in B_{\varepsilon}(\eta)$.
Then, it follows from Lemma \ref{lem:distance} that $\eta'$ should be constant in the interval $[\varepsilon, 1-\varepsilon]$, that is $\T'_r= \X'$ for all $r \in [\varepsilon ,  1-\varepsilon]$, where $\X'$ is a torsion class.
As otherwise, there exists a $t$ in that interval for which there is a non-zero object $M \in\P'_t$, contradicting the assumption that $d(\eta, \eta') < \varepsilon$.

If $\X'=\X$, there is nothing to prove.
Suppose that $\X'$ is different from $\X$. 
Then we can suppose without loss of generality the existence of an object $M \in \X$ that is not in $\X'$.
Now, consider the Harder-Narasimhan filtration 
$$M'_0 \subset M'_1 \subset \dots \subset M'_n$$
of $M$ with respect to $\eta'$ given by Theorem \ref{thm:HN-filt}.
Because $M \not \in \X'$ we have that $\mho_{\eta'} (M)  \in [0,\varepsilon]$. 
However, $\mho_{\eta}(M)=\Omega_\eta(M)=1$ since $M \in \X$.
Hence, $\X'\neq \X$ implies that $\eta' \not \in B_{\varepsilon}(\eta)$.
That finishes the proof.
\end{proof}

\subsection{Construction and topological properties of the space of slicings} \label{sec:top-properties}

Any chain of torsion classes $\eta$ can be seen as an order reversing function $\eta: [0,1] \to \Tors \A$, with respect to  their natural poset structures. 
We call the image $\im \eta \subset \Tors \A$ a \textit{$n$-sequence} of $\Tors \A$ if the cardinality of $\im \eta$ is $n$ for some cardinal $0 < n  \leq \mathfrak{c}$ where is the cardinal of $\mathbb{R}$.
For a sequence $S= \im \eta$, we denote by $[n]$ the indexing set of $S$ with the total order induced by $\eta$. 
We summarise the data of an $n$-sequence $S$ as follows:
$$S: \qquad \A = \X_0 \supsetneq \X_i \supsetneq \X_j\supsetneq \X_n = {0} \text{ for every $i < j \in [n]$}.$$
Note that two chains of torsion classes $\eta, \eta'$ induce the same sequence in $\Tors \A$ if and only  there exists a continuous order-preserving bijection $f: [0,1] \to [0,1]$ such that $\P_{f(t)}= \P'_{t}$ for all $t$ in $[0,1]$. This is an equivalence relation on the set $\fT(\A)$ and we study this further in Section~\ref{sec:wall-and-chamber}.

We say an $n'$-sequence $S'$ is a \textit{subsequence} of an $n$-sequence $S$ if there exists an injective increasing function $f: [n'] \to [n] $ such that $\X'_m = \X_{f(m)}$ for all $m \in [n']$. Note if such an $f$ exists then it necessarily unique. 

For convenience, we say that an $n$-sequence, $S$, in $\Tors \A$ is \emph{well-ordered} if the natural order on the set $[n]$ is a well-order.
    
For each $n$-sequence $S$ in $\Tors\A$ we take an $n$-simplex $\Delta_S = \{ (x_1, \dots, x_n) \in \mathbb{R}^n \mid 0 \leq x_1 \leq x_2 \leq \dots \leq x_n \leq 1 \}$ where $\mathbb{R}^n$ denotes the product of $n$ copies of $\mathbb{R}$.
From each point in the interior of $\Delta_S$ we obtain a chain of torsion classes as follows: Let $\bx = (x_1, \dots, x_n) \in \Delta_S^\mathrm{o}$ for some $n$-sequence $S: \, \A = \X_0 \supsetneq \X_1 \supsetneq \dots \supsetneq \X_n = {0}.$ Then we set $h(\bx)$ to be the chain of torsion classes $(\T_i)_{i \in [0,1]}$ defined by
\[ \T_i = 
\begin{cases}
\A \text{ if $i \in [0, x_1]$}\\
\X_j \text{ if $i \in (x_j, x_{j+1}]$ for $1\leq j < n$}\\
\0 \text{ if $i \in (x_n, 1].$}
\end{cases} \]
Note that $\eta^{-}\mho_{h(\bx)} = h(\bx)$. 

The next result tells us that the weakly-Artinian and quasi-Noetherian properties of chains of torsion classes depends only on the underlying sequence of torsion classes.

\begin{proposition}
Let $S$ be an $n$-sequence in  $(\Tors \A, \subset)$. Suppose there exists a point $\bx$ in $\Delta_S^\mathrm{o}$ such that $h(\bx)$ is weakly-Artinian (resp. quasi-Noetherian), then $h(\bx')$ is also weakly-Artinian (resp. quasi-Noetherian) for all $\bx' \in \Delta_S^\mathrm{o}$.
\end{proposition}
\begin{proof}
Let  $\bx = (x_1, \dots, x_n)$ and $\bx' = (x'_1, \dots, x'_n) $ be points in $\Delta_S^\mathrm{o}$. Observe that there exists an order preserving injection $\psi : [0,1] \to [0,1]$ such that $\psi (x_i) = x'_i$ for all $i \in [n]$.

Let $(a,b) \subset [0,1]$ be an interval. 
Then the sequence 
\[ M \supseteq t_a M \supseteq \dots \supseteq t_b M\]
stabilises below (resp. above) if and only if the sequence
\[ M \supseteq t'_{\psi(a)} M \supseteq \dots \supseteq t'_{\psi (b)} M \]
also stabilises below (resp. above). 
As consequence $h(\bx)$ is weakly-Artinian (resp. quasy-Noetherian) if and only if so is $h(\bx')$.
\end{proof}

We may now describe a topological space $\mathfrak{M}(\A)$: For each weakly-Artinian and quasi-Noetherian $n$-sequence, $S$, in $\Tors\A$ we take an $n$-simplex $\Delta_S$ as before.
If $S'$ is a $n'$-subsequence of $S$ we identify $\Delta_{S'}$ with subset of $\Delta_S$ given by
\[ \{ (x_1, \dots, x_n) \in \Delta_S \mid x_j = \inf\{x_{f(i)} \mid f(i)> j\}\}\] 
where $f: [n'] \to [n]$ is the unique injective increasing function described above. 
Observe that if $S'$ is well ordered then $\Delta_{S'}$ is identified with 
the face of $\Delta_S$ given by 
\[ \{ (x_1, \dots, x_n) \in \Delta_S \mid  f(i-1) < j \leq f(i) \text{ in } [n] \Rightarrow x_j = x_{f(i)} \} \] 
Note that clearly, if $S$ is  weakly-Artinian and quasi-Noetherian then any subseqence of $S$ also has these properties.

\begin{remark}\label{rem:slices-space-properties}
We make some observations. 
\begin{enumerate}
\item If all sequences in $\Tors \A$ are well-ordered (for instance, if $\Tors \A$ is finite), then $\mathfrak{M}(\A)$ may also be constructed as the classifying space or nerve of $\Tors \A$.
\item $\mathfrak{M}(\A)$ inherits the structure of a (psuedo)metric space from the psuedometric of $\mathfrak{T}(\A)$ that was discussed in Section~\ref{sec:spacetor} if we define $d(\bx, \bx') := d(h(\bx), h(\bx'))$. 
\item For $\bx, \bx' \in \Delta_S$ we have that $d(\bx, \bx') = \sup \{ |x_i - x_i'| \mid i \in [n] \} $. In particular, if $n$ is well-ordered, the metric on $\Delta_S$ coincides with the Chebyshev distance. That is, for $\bx, \bx' \in \Delta_S$,  
$d(\bx, \bx') = \max\{ |x_i - x_i'| \mid i \in [n] \}$. 
\item The poset $\Tors \A$ is finite if and only if there are finitely many weakly-Artinian and quasi-Noetherian sequences in $\Tors \A$.
\end{enumerate}

\end{remark}

From our construction and Proposition~\ref{prop: slice(A)  cong T(A)} we obtain the following:

\begin{proposition}
The space $\mathfrak{M}(\A)$ is isomorphic to  $\fT(\A) / \sim$ and hence also to $\Slice (\A) $.
\end{proposition}

We now study some topological properties of $\Slice (\A)$ and $\fT(\A)$. We begin by recalling that two points $p$ and $q$ in a topological space $X$ are said to be \textit{path connected} if there exists a continuous function $f: [0,1] \to X$ such that $f(0)= p$ and $f(1)=q$.
We claim that $\Slice (\A)$ is path connected. 
Indeed, every simplex in the space has a face corresponding to the sequence $\A \supset 0$. 
Moreover, since the psuedometric and hence the topology of $\fT(\A)$ only depends on the slicings of points, the path connectedness of $\Slice (\A)$ implies the path connectedness of $\fT(\A)$. 

The following observation will be useful in the sequel.

\begin{remark} 
Let $S = \{\X_i\}_{i \in [n]}$ be an $n$-sequence in $\Tors \A$. 
Then, for all $\bx, \bx' \in  \Delta^{\mathrm{o}}_S$, we have that $\cP^{h(\bx)}_{x_i} = \cP^{h(\bx')}_{x_i'}$ for all $i \in [n]$. 
In the sequel, we set  $\cP^S_i := \cP^{h(\bx)}_{x_i}$ for $\bx \in \Delta^{\mathrm{o}}_S$. 
We also note that, if $S$ is well-ordered then  $\cP_i^S  = \X_{i-1} \cap \X^\perp_i$ for $i >1$ and $\cP_1^S = \X^\perp_1$.
\end{remark}

We have a useful alternative characterisation of subsequences of well-ordered sequences in $\Tors \A$. 

\begin{lemma} \label{lem:subsequences} 
Let $S$ (resp. $S'$) be an $n$-sequence (resp. $n'$-sequence) in $\Tors \A$ and suppose that $S$ is well-ordered. 
Then $S'$ is a subsequence of $S$ if and only if there exists an order preserving surjective function $g: [n] \to [n']$ such that $\cP^S_\alpha \subseteq \cP^{S'}_{g(\alpha)}$ for all $\alpha \in [n]$.  
\end{lemma}

\begin{proof}
First suppose that $S' = \{\X'_i\}_{i \in [n']}$ is a subsequence of $S = \{\X_i\}_{i \in [n]}$. 
Then, by definition, there exists an injective order preserving fuction $f: [n'] \to [n]$ such that $\X'_\beta = \X_{f(\beta)}$ for all $\beta \in [n']$ furthermore as $[n]$ is well-ordered so is $[n']$. 
Then it follows that for all $\alpha \in [n]$ there exists a unique $\beta \in [n']$ such that $\alpha \leq f(\beta)$ and $\alpha > f(\beta ')$ for all $\beta > \beta' \in [n']$. We set $g(\alpha) = \beta$ and claim that this assignment satisfies the conditions of the statement. 
Indeed, $g$ is clearly order preserving and as $gf(\beta) = \beta$ for all $\beta \in [n']$ we deduce that $g$ is surjective. Finally, it remains to show that $\cP_\alpha^S \subseteq \cP^{S'}_{g(\alpha)} $.
By the definition of $g(\alpha)$ we have that $\X_\alpha \subseteq \X_{fg(\alpha)} = \X_{g(\alpha)}'$, thus $\bigcap_{r<\alpha} \X \subseteq \bigcap_{r<g(\alpha)}\X'$ and $\bigcup_{r> \alpha} \X \subseteq \bigcup_{r>g(\alpha)} \X'$ and we are done by Definition~\ref{def:slicing}. 

For the converse, let $g$ be such a surjective order preserving function. 
Since $[n]$ is well-ordered, it then follows that for all $\beta \in [n']$ that there exists unique $\alpha \in [n]$ such that $g(\alpha) = \beta$ and $g(\alpha') > \beta$ for all $\alpha < \alpha' \in [n]$. By setting $f(\beta) = \alpha$ we define an order preserving injective function. 

It remains to show that $\X'_\beta = \X_{f(\beta)}$ for all $\beta \in [n']$. 
By assumption,  for all $f(\beta) < i \in [n]$ we have that $\cP_i^S \subseteq \cP^{S'}_{g(i)}$ and $\beta < g(i)$. 
Thus 
\[\X_{f(\beta)} = \Filt \left( \bigcup_{f(\beta) < i \leq n} \cP^{S}_j \right) \subset \Filt \left( \bigcup_{\beta < j \leq n'} \cP^{S'}_j \right)  = \X'_\beta. \]
For the reverse inclusion, suppose that there exists $M \in \X'_\beta \setminus \X_{f(\beta)}$.  
Then $\mho_S(M) = \sup \{ i \in [n] \mid M \in \X_i\} \leq f(\beta)$. 
In particular, there exists a quotient $N$ of $M$ such that $N \in \cP^S_i$ for some $i \leq f(\beta)$. 
By Theorem~\ref{thm:dicksontorsion}.2, $N \in \X'_\beta$. 
We deduce that $g(i) \geq \beta$ which is a contradiction. 
\end{proof}

\begin{remark}
Let $S$ be a well-ordered $n$-sequence and take $S'$ a $n'$-subsequence of $S$. 
Let $g: [n] \to [n']$ be the order preserving surjection as in the above lemma.
Then it follows from the proof of the that lemma that 
$$\P'_\beta = \Filt \left( \bigcup_{g(\alpha)=\beta}  \P_\alpha\right) $$
for all $\alpha \in [n]$ and $\beta \in [n']$.
\end{remark}

We now show that every well-ordered sequence in $\Tors \A$ determines a closed set of $\Slice \A$. 

\begin{proposition} \label{prop: Delta_S closed}
Let $S= \{\X_i\}_{i \in [n]} $ be a well-ordered $n$-sequence in $\Tors \A$. 
Then $\Delta_S$ is a closed set of $\Slice \A  $. 
\end{proposition}
\begin{proof}
We show that $\Delta_S$ contains all of its accumulation points. 
To do this, it is enough to show that all accumulation points of $\Delta^{\mathrm{o}}_S$ lie in $\Delta_S$. 
To this end, let $\by$ be such an accumulation point, so that there exists a sequence  $(\bx_i)_{i \in \mathbb{N}} \in  \Delta^{\mathrm{o}}_S$ such that $d(\by, \bx_i) < 1/i $. 
We may assume that $\by \in \Delta^{\mathrm{o}}_{S'}$ for some $n'$-sequence $S' = \{\X'_i\}_{i \in [n']}$.
We will show that $S'$ must be a subsequence of $S$ using the characterisation given in Lemma~\ref{lem:subsequences}. 

To this end, we first show that for all $\alpha \in [n]$, there exists $g(\alpha) \in [n']$ such that $\cP_\alpha^S \subseteq \cP_{g(\alpha)}^{S'}$. 
Suppose that this is not the case for $\alpha \in [n]$. 
Then we consider the minimal interval $J \subseteq [n']$ such that $\cP_\alpha^{S} \subseteq \cP^{S'}_J$. 
By assumption $|J| \geq 2$, we take distinct $j, j' \in J$. Then by Lemma~\ref{lem:distance}, we have that $d(\by, \bx) \geq \frac{|y_j - y_{j'}|}{2}$ for all $\bx \in \Delta^{\mathrm{o}}_S$ which is clearly a contradiction to $\by$ being an accumulation point of $\Delta^{\mathrm{o}}_S$.

We must now show that the assignment $\alpha \mapsto g(\alpha)$ is order preserving and surjective, whence we will be finished by Lemma~\ref{lem:subsequences}. 
Suppose first that $g$ is not order preserving, then there exists $\alpha <\beta \in [n]$ such that $g(\alpha) > g(\beta)$. 
Then, by Lemma~\ref{lem:distance}, we have for all $\bx \in \Delta^{\mathrm{o}}_S$ that 
\[d(\by, \bx) \geq \max \{ |y_{g(\alpha)} -x_\alpha |, |y_{g(\beta')} - x_\beta | \} \geq \frac{(y_{g(\alpha)} - y_{g(\beta)})}{2}.\] 
Since this distance is bounded below by a constant, we conclude that $\by$ cannot be an accumulation point of $\Delta_S$. 

Finally, suppose that $g$ is not surjective and let $\beta \in [n']$ be not in the image of $g$. 
Let $X \in \cP^{S'}_\beta$ be non-zero, then the Harder-Narasimhan filtration of $X$ with respect to $h(\by)$ is trivial, that is, it is of the form $0 \subset X$. 
Let $\bx \in \Delta^{\mathrm{o}}_S$, then if $X$ is $h(\bx)$-quasisemistable, that is $X \in \cP^S_\alpha$ for some $\alpha$ we have a contradiction to the assumption that $\cP^S_\alpha \subseteq \cP^{S'}_{g(\alpha)}$. 
If $X$ is not $h(\bx)$-quasisemistable then its Harder-Narasimhan filtration with respect to $h(\bx)$ is non-trivial. 
By construction of $g$ we have that every $h(\bx)$-quasisemistable object is $h(\by)$-quasisemistable, moreover, as $g$ is order preserving we have that the Harder-Narasimhan filtration of $X$ with respect to $h(\bx)$ is also a Harder-Narsimhan filtration of $X$ with respect to $h(\by)$. 
This contradicts the uniquess of Harder-Narasimhan filtrations as shown in Theorem~\ref{thm:HN-filt}. 
\end{proof}

\subsection{Abelian categories with finitely many torsion classes}
In this subsection we study in more detail the case when $\Tors \A$ is a finite. 

\begin{proposition} \label{prop: finite-tors-weak-topology}
Suppose that $\Tors \A$ is a finite set. Then a set $V \subseteq \Slice \A$ is closed if and only if $V \cap \Delta_S$ is closed in $\Delta_S$ for all sequences $S$ in $\Tors \A$.
\end{proposition}
\begin{proof}
Since $\Tors \A$ is finite, we may write any subset, $V$, of $\Slice \A$ as a finite union $V = \bigcup_S (V \cap \Delta_S)$. 
Thus if $V \cap \Delta_S$ is closed for all $S$ then so is $V$. Conversely, as $\Tors \A$ is finite every $n$-sequence is well-ordered thus $\Delta_S$ is closed in $\Slice \A$ for all $S$ by Proposition~\ref{prop: Delta_S closed}. 
Now, clearly if $V$ is closed then so is $V \cap \Delta_S$.  
\end{proof}

\begin{corollary} \label{cor: finite-tors-then-finite CW}
Suppose that $\Tors \A$ is finite. Then $\Slice \A$ is homeomorphic to a finite CW complex. 
\end{corollary}
\begin{proof}
One may clearly interpret the construction of $\mathfrak{M}(\A)$ as a simplicial set which is finite if $\Tors \A$ is. 
By Proposition~\ref{prop: finite-tors-weak-topology} if $\Tors \A$ is finite, then the topology of $\Slice \A$ coincides with the standard weak topology of a CW complex. 
\end{proof}

\begin{remark}
It is inconsequential that the metric in each $\Delta_S$ is not the standard Euclidean metric. 
We recall that all norms on a finite dimensional vector space are equivalent and so are their induced metrics. 
Moreover, equivalent metrics induce equivalent topologies.
As we noted in Remark~\ref{rem:slices-space-properties}, the metric within each $\Delta_S$ is the Chebyshev distance, which is induced by the uniform norm and is thus equivalent to the usual Euclidean metric which is induced by the Euclidean norm. 
\end{remark}

From the above we give a characterisation of abelian categories with finitely many torsion classes.

\begin{theorem}\label{thm:compact}
The space $\Slice \A$ is compact if and only if 
$\Tors \A$ is a finite set. 
\end{theorem}

\begin{proof}
Suppose that $\Tors \A$ is finite. 
Then by Corollary~\ref{cor: finite-tors-then-finite CW}, $\Slice \A$ is equivalent to a finite CW complex and is therefore compact. 

For the converse, suppose that $\Tors \A$ is infinite. 
For a torsion class $\T \in \Tors \A$ consider the chain of torsion classes $\eta(\T) = (\X)_{i \in [0,1]}$ defined by 
\[ \X_i =  \begin{cases}
\A \text{ if $i = 0$,}\\
\T \text{ if $i \in (0,1)$, }\\
\0 \text{ if $i =1.$}
\end{cases} \]
Observe that for all $\T \neq \T' \in \Tors \A$ by Lemma~\ref{lem:distance}, we have that $d(\eta(\T), \eta(\T')) = 1$. 
Consider the collection of open sets $\U = \{ B_{\frac{1}{2}}(\eta(\T)) \mid \T \in \Tors \A \}$, we complete this to an open cover of $\Slice \A$ by adding open balls of radius $\varepsilon < \frac{1}{2}$ for each point not covered by $\U$. 
By construction this is an infinite cover of $\Slice \A$ that admits no finite subcover, thus we deduce that $\Slice \A$ is non-compact. 
\end{proof}

An important example of abelian categories with finitely many torsion classes are the module categories of $\tau$-tilting finite algebras, a family of algebras introduced in \cite{DIJ}, see also \cite{TreffingerSurvey} for more details. 
Our results give a new characterisation of these algebras in terms of the slicings of their module categories. 

\begin{corollary} \label{cor:taufinite}
Let $\Lambda$ be a finite dimensional algebra and set $\A = \mod \Lambda$. 
Then $\Slice \A$ is compact if and only if $\Lambda $ is $\tau$-tilting finite. 
\end{corollary}

\subsection{Distinguished open and closed sets}

In this final subsection we study some important closed and open subsets of $\Slice \A$ and $\fT(\A)$. 

Fix a nonzero object $X \in \A$ and consider the set
\[
L(X) = \left\{ \P \in \Slice \A \mid X \in \bigcup_{i\in [0,1]} \P_i\right\}.
\]
These subsets are related with the so-called \textit{walls of type II} introduced in \cite{Bridgeland2007}. 
One of the reasons behind their name is explained in the following proposition.

\begin{proposition}\label{prop:wallII}
    The set $L(X)$ is closed in $\Slice\A$ for every nonzero object $X \in \A$.
\end{proposition}

\begin{proof}
Let $\by \in \Slice\A$ be an accumulation point of $L(X)$. 
By Proposition~\ref{prop:metric} we have that
\[ d(\bx, \by) = \sup \{ |\Omega_{h(\by)}M - \Omega_{h(\bx)}M |, | \mho_{h(\by)}M - \mho_{h(\bx)}M | \mid 0 \neq M \in \A\} \]
for all $\bx \in \Slice \A$. 
If $\bx \in L(X)$ we have that $\mho_{h(\bx)}X = \Omega_{h(\bx)}X$. 
Thus in this case $d(\bx, \by)$ is bounded below by $\frac{\Omega_{h(\by)}X - \mho_{h(\by)}X}{2}$.
If $\by$ is not an element of $L(X)$ then $X$ is not $h(\by)$-quasisemistable. 
In particular, $\Omega_{h(\by)}N > \mho_{h(\by)}N$ which is clearly an obstruction to $\by$ being an accumulation point of $\Slice(X)$.
\end{proof}

By $\Slice^{\tot}\A$ we denote the space of \emph{split} slicings. 
That is, all slicings $\cP$ on $\A$ such that every indecomposable object $M$ in $\A$ there exists $r \in [0,1]$ such that $M \in \cP_r$. 
A direct consequence of Proposition~\ref{prop:wallII} we obtain the following.

\begin{corollary}
$\Slice^{\tot}\A$ is closed in $\Slice \A$.    
\end{corollary}

\begin{proof}
It is enough to see that 
\[
\Slice^{\tot}\A = \bigcap_{0 \neq X\in \operatorname{ind}\A } L(X)
\]  
and thus we are done since the right hand side is an intersection of closed sets by Proposition~\ref{prop:wallII}.
\end{proof}

Following \cite{tattar2021torsion}, we call torsion pairs $(\X, \Y)$ and $(\X', \Y')$  in $\A$ such that $\X \subseteq \X'$ \emph{twin}. 
In the following result, which is a generalisation of \cite[Theorem 6.11]{tattar2021torsion}, we see that twin torsion pairs induce closed sets.

\begin{proposition} \label{prop:heartclosed}
Let $(\X, \Y)$ and $(\X', \Y')$ be twin torsion pairs in $\A$ and let $(a,b) \subset [0,1]$ be an interval. 
Then the set 
\[ Z = \left\{ \bx \in \Slice \A \mid  \T_i = 
    \begin{cases}
    \A \text{ if $i \in [0,a)$,}\\
\T \text{ with $\X \subseteq \T \subseteq \X'$ if $i \in (a,b)$, }\\
\0 \text{ if $i \in (b,1].$}
    \end{cases} \text{ where } h\bx = (\T_i)_{i \in [0,1]}  
     \right\}  \] 
is closed in $\fT(\A)$. 
\end{proposition}
\begin{proof}
Note that the set $Z$ coincides with the set 
\begin{equation} \label{eqn:twinslicing}
    \{\cP \in \Slice \A \mid \cP_a \subseteq \Y', \, \cP_b \supseteq \X, \text{ and } \cP_r = 0 \forall r \in [0,a) \cup (b,1] \}.  
\end{equation} 
Now let $\by \in \Slice \A$ be an accumulation point of $Z$, we will show that the slicing corresponding to  $\by$ satisfies the criteria of the set \ref{eqn:twinslicing}. 
First suppose that there exists $s \in [0,a) \cup (b,1]$ such that $\cP_s^{h(\by)} \neq \{0\}$. 
Then by Lemma~\ref{lem:distance}, we have for all $\bx \in Z$ that \[ d(\by, \bx) \geq \min \left\{ \frac{|a-s|}{2}, \frac{|b-s|}{2} \right\} \] which is a contradiction to $\by$ being an accumulation point of $Z$. 
Now we show that $\X \subset \cP^{h(\by)}_b$. Suppose that this is not the case, and let 
\[  s = \sup\left\{ i \in [0,1] \mid \X \subseteq \Filt \left( \bigcup_{r \in [s,b]} \cP^{h\by}_r \right) \right\}.\]
Note that by assumption $s<b$. 
Then, once more by Lemma \ref{lem:distance}, we have that $ d(\by, \bx) \geq \frac{|b-s|}{2}$ for all $\bx \in Z$ which is a contradiction to $\by$ being an accumulation point. 
Using a similar argument, one shows that $\cP^{h(\by)}_a \subseteq \Y'$. 
\end{proof}

\begin{remark} \label{rem:bricklocus}
Let $\A=\mod A$ for a finite-dimensional algebra $A$ and let $\X \subset \X'$ be a minimal inclusion in $\Tors \A$, then it follows from \cite[Theroem~2.8]{BarnardCarrollZhu} that $\X' = \Filt ( \X \cup \{B\} ) $ for some brick $B \in \A$, this is also known as the \emph{brick label} \cite{BarnardCarrollZhu, DIRRT} of the arrow corresponding to the covering relation $\X' > \X$ in the Hasse quiver of $\Tors \A$. 
Then the sets $Z$ as considered in Proposition~\ref{prop:heartclosed} are subsets of $L(B)$. 
\end{remark}

\subsubsection{Wall-and-Chamber structure of \texorpdfstring{$\CT(\A)$}{CT(A)}} \label{sec:wall-and-chamber}

As we noted at the start of Section~\ref{sec:top-properties}, there is another natural equivalence relation in $\CT(\A)$ which is stronger than the one considered in Section~\ref{section:equivrelation} as follows.
Let $\eta, \eta'\in \CT(\A)$ and $\P_\eta$, $\P'_{\eta'}$ their respective slicings.
We say that $\eta$ is equivalent to $\eta'$ if there exists a continuous order-preserving bijection $f: [0,1] \to [0,1]$ such that $\P_{f(t)}= \P'_{t}$ for all $t$ in $[0,1]$. 
Or equivalently, $\eta$ is equivalent to $\eta'$ if they induce the same sequence in $\Tors \A$. 

We denote the partition of $\CT(\A)$ induced by the equivalence relation of the previous definition as the \textit{wall-and-chamber structure} of $\CT(\A)$.
This terminology, borrowed from \cite{Bridgeland2007}, comes from the analysis of the interaction between the cosets of the equivalence class with the topology of $\CT(\A)$.

Note that this new equivalence relation is stronger than the equivalence relation $\sim$ that we have previously considered. Indeed, if $\eta \sim \eta'$ then the identity $f: [0,1] \to [0,1]$ shows that $\eta$ and $\eta'$ are equivalent.  Thus, one can also construct the the wall-and-chamber structure of $\fT(\A)$ from the space $\Slice \A = \fT(\A) / \sim$.

The first thing to remark is that every coset of $\CT(\A)$ is connected. 
A coset of $\CT(\A)$ is said to be a \textit{chamber} if it is an open set of $\CT(\A)$. 
Otherwise, we say that the coset is a subset of a \textit{wall}. 
The choice of the word \textit{wall} is explained by the fact that it was shown in \cite{Bridgeland2007} that every wall in $\Stab(\mathcal{D})$, the space of Bridgeland stability conditions of a triangulated category $\mathcal{D}$, is of a closed space of codimension $1$.

Recall \cite{Keller2011b, BSTw&c} that a finite sequence of torsion classes $\A = \T_0 \subset \T_1  \subset \dots \subset \T_m = \{0\}$ is a \emph{maximal green sequence} if for all $1\leq i \leq m$ and torsion classes $\X$ such that $\T_{i-1} \subset \X \subset \T_i$, we have that $\X \in \{ \T_{i-1}, \T_i \}$.
For more details on maximal green sequences see \cite{KellerSurvey}.
We now show maximal green sequences in $\A$ determine the chambers of $\CT(\A)$.

\begin{theorem}\label{thm:chambers}
Then there is a one-to-one correspondence between the chambers of the space $\CT(\A)$ and the maximal green sequences in $\A$.
\end{theorem}

\begin{proof}
The statement of the theorem is equivalent to saying that $\eta \in \CT(\A)$ is a maximal green sequence if and only if there exists an $\varepsilon > 0$ such that for all $\eta' \in B_{\varepsilon}(\eta) $ the Harder-Narasimhan filtrations of $M$ with respect to $\eta$ and $\eta'$ are isomorphic for all $M \in \Obj^\ast \A$. 

We first show that if $\eta$ has a neighbourhood as above, then $\eta$ is a maximal green sequence. 
Suppose that $\eta \in \CT(\A)$ is a chain of infinitely many torsion classes. 
Then, this implies that the set $S=\{x \in [0,1] \mid \P_x \neq \{0\}\}$ is infinite.
Now, given that $[0,1]$ is a compact space with the respect to the usual topology, we have that there exists a point $y \in [0,1]$ which is an accumulation point of $S$. 
That is, for every $\varepsilon > 0$ there exists infinitely many $x \in (y - \varepsilon, y+ \varepsilon)$ such that $\P_x \neq \{0\}$.
Then, starting from $\eta$, we construct a chain of torsion classes $\eta' = (\T'_i)_{i \in [0,1]} \in \CT(\A)$ by 
$$\T'_r= \begin{cases}
\T_r \text{ if $r \in [0, y-\varepsilon/2)$} \\
\T_{y+\varepsilon/2} \text{ if $r \in [y-\varepsilon/2, y+ \varepsilon/2]$}\\
\T_r \text{ if $r \in  [y+\varepsilon/2 , 1]$}. 
\end{cases}$$
It follows from Lemma \ref{lem:distance} that $\eta' \in B_{\varepsilon}(\eta)$.
On the other hand, there are  $x_1, x_2 \in (y- \varepsilon/2 , y+\varepsilon/2)$ such that $\P_{x_1} \neq \{0\}$ and $\P_{x_2} \neq \{0\}$.
We can suppose without loss of generality that $x_1 < x_2$.
Let $M_1 \oplus M_2$ be the direct sum of $M_1$ and $M_2$, a non-zero object in $\P_{x_1}$ and $\P_{x_2}$, respectively. 
Then, it is easy to see the filtration of $M$ with respect to $\eta$ is simply $0 \subset M_2 \subset M_1 \oplus M_2$.

We also have that $M_1 \in \P_{x_1}=  \bigcap\limits_{s<x_1} \T_s  \cap \bigcap\limits_{s>x_1} \F_s $. 
Since $x_1 > y - \varepsilon/2$ we have that $\bigcap\limits_{s>x_1} \F_s \subset \bigcap\limits_{y - \varepsilon/2 > s}\F_s$.
Hence $M_1 \in \bigcap\limits_{y - \varepsilon/2 > s}\F_s=  \bigcap\limits_{y - \varepsilon/2 > s}\F'_s$.
Moreover, from the fact $x_1 < y + \varepsilon/2$ we can conclude $\bigcap\limits_{s<y + \varepsilon/2} \T_s  \subset\bigcap\limits_{s<y + \varepsilon/2} \T'_s$.
So, $M_1 \in \bigcap\limits_{s<y + \varepsilon/2} \T_s = \bigcap\limits_{s<y + \varepsilon/2} \T'_s$ by construction of $\eta'$.
Therefore $M_1 \in \P'_{y+\varepsilon/2}$.
Likewise we can prove that $M_2 \in \P'_{y+\varepsilon/2}$.
Hence $M_1 \oplus M_2 \in \P'_{y+\varepsilon/2}$.
Then the filtration of $M_1 \oplus M_2$ with respect to $\eta'$ is simply $0 \subset M_1\oplus M_2$, which is different from the filtration of $M_1 \oplus M_2$ with respect to\;$\eta$.
Therefore a chain of torsion classes $\eta \in \CT(\A)$ having an open neighbourhood as in the statement is a chain of finitely many different torsion classes in $\A$.
That is for every $r \in [0,1]$ we have that $\eta$ is, up to slicing, of the form 
$$\T_r = 
\begin{cases}
\X_0 = \A &\text{ if $r \in [0, a_1)$}\\
\X_1 &\text{ if $r \in [a_1, a_2)$}\\
\vdots \\
\X_{n-1} &\text{ if $r \in [a_{n-1}, a_{n})$}\\
\X_n = \{0\} &\text{ if $r \in [a_{n}, 1]$}
\end{cases}$$
where $\X_i$ is a torsion class in $\A$ for all $i$ and $ \A = \X_ 0 \supsetneq \X_1 \supsetneq \dots  \supsetneq X_n = \{0\}$. 
In particular, one can see that $\P_x \neq \{0\}$ if and only if $x=a_i$ for some $i$.

Suppose that there exists a torsion class $\X$ in such that $\X_i \subsetneq \X \subsetneq \X_{i-1}$. 
Then we can construct a chain of torsion classes $\eta' = (\T'_r)_{i\ in [0,1]} \in \CT(\A)$ as follows
$$\T'_r = 
\begin{cases}
\X_0 \text{ if $r \in (0, a_1)$}\\
\vdots \\
\X_{i-1} \text{ if $r \in [a_{i-1}, a_{i-1} + \delta)$}\\
\X \text{ if $r \in [a_{i-1}+\delta , a_{i})$}\\
\X_i \text{ if $r \in [a_{i}, a_{i+1})$}\\
\vdots \\
\X_n \text{ if $r \in [a_{n}, 1]$}
\end{cases}$$
where $\delta < \min \{\frac{\varepsilon}{2}, \frac{a_{i}- a_{i-1}}{2}\}$. 
Now, we can calculate $\P'_x$ for every $x \in [0,1]$ and see that $\P'_x \neq 0$ if and only if $x=a_i$ for some $i$ or $x=a_{i-1} + \delta$.
Therefore, it follows from Lemma \ref{lem:distance} that $d(\eta, \eta') < \varepsilon$. 
Moreover one can repeat the argument that we use before to show that there exists $M_1 \in \P'_{a_{i-1}}$ and $M_{2} \in \P'_{a_{i-1}+\delta}$ such that the filtration of $M_1 \oplus M_2$ with respect to $\eta$ is not isomorphic to its filtration with respect to $\eta'$. 
So, we can conclude that if $\eta$ is a chain of torsion classes having a neighborhood $B_{\varepsilon}(\eta)$ as in the statement $\eta$ should be finite and it should not admit any refinements; in other words,  
$\eta$ is a maximal green sequence, as claimed.

\medskip
Now we take a maximal green sequence $\eta \in \CT (\A)$ and we show  the existence of a neighbourhood $B_\varepsilon(\eta)$ as in the statement.
Since $\eta$ is a maximal green sequence, we have that every $\T_s\in \CT$ is, up to slicing, of the form 
$$\T_r = 
\begin{cases}
\X_0 \text{ if $r \in (0, a_1)$}\\
\X_1 \text{ if $r \in [a_1, a_2)$}\\
\vdots \\
\X_n \text{ if $r \in [a_{n}, 1)$}
\end{cases}$$
where $\X_i$ is a torsion class for every $i$, $\X_0=\A$, $\X_n = \{0\}$ and given a torsion class $\X$ such that $\X_{i} \subset \X \subset \X_{i-1}$ then $\X$ is either equal to $\X_i$ or $\X_{i-1}$.
Let 
$\varepsilon < \min \left\{\frac{a_{i+1} - a_{i}}{2} \mid 1 \leq i \leq n-1 \right\}$ and let $\eta' \in B_\varepsilon(\eta)$ with slicing $\P_{\eta'}$. 
Let $\eta'' = (\T_r'')_{r \in [0,1]}$ be the be the chain of torsion classes   such that for all $r \in (0,1)$ 
$$\T''_{r} = \Filt \left( \bigcup_{t>r} \P'_t \right).$$
It is easy to see that $d(\eta',\eta'')=0$ since $\P_{\eta'}= \P_{\eta''}$. 
Then Lemma \ref{lem:distance} implies that 
$$\X_i=\T_{a_i}=\Filt \left( \bigcup_{t>a_i} \P_t \right) \subset \Filt \left( \bigcup_{t>a_i-\varepsilon} \P'_t \right) = \T''_{a_i- \varepsilon}.$$
On the other hand, 
$$\T''_{a_i-\varepsilon} = \Filt \left( \bigcup_{t>a_i-\varepsilon} \P'_t \right) \subset \Filt \left( \bigcup_{t>a_{i}-2\varepsilon} \P_t \right) = \T_{a_{i}-2\varepsilon}=\X_{i-1}.$$
Thus, as $\eta$ is a maximal green sequence,  $\T''_{a_i - \varepsilon} \in \{ \X_i, \X_{i-1} \}$. Since $d(\eta, \eta'') < \varepsilon$ we deduce that $\T''_{a_i - \varepsilon} = \X_{i}$. 
Again appealing to the fact that $\X_0 \supsetneq \X_1 \supsetneq \dots \supsetneq \X_n$ is a maximal green sequence, we deduce that $\T''_r \in \{ \X_i \mid 0 \leq i \leq n \} $ for all $r \in [0,1]$. 
Finally, as $d(\eta, \eta'') < \varepsilon$ we conclude that, up to slicing, $\eta''$  of the form 
$$\T''_r = 
\begin{cases}
\X_0 \text{ if $r \in [0, b_1)$}\\
\X_1 \text{ if $r \in [b_1, b_2)$}\\
\vdots \\
\X_n \text{ if $r \in [b_{n}, 1]$}
\end{cases}$$
where $b_i$ is such that $|a_i-b_i|< \varepsilon$ for all $1 \leq i \leq n$.
\end{proof}

To finish, we also note that one may use the language we have developed to add to the characterisations of torsion classes in Theoerem~\ref{thm:dicksontorsion}.

\begin{proposition}
    The data of a proper torsion class in $\A$ is equivalent to a $\phi \in \WStab \A$ such that $|\left\{ i \in [0,1] \mid \cP^\phi_i \neq \{0\} \right\} | = 2$. 
    Moreover, two such weak stability conditions induce the same torsion class if and only if they are equivalent.
\end{proposition}

\def\cprime{$'$} \def\cprime{$'$}
\addcontentsline{toc}{section}{References}
\bibliographystyle{plain}

 \bigskip
  \begin{tabular}{@{}l@{}}%
  \textsc{Abteilung Mathematik, Department Mathematik/Informatik der Universit\"at } \\ \textsc{zu K\"oln, Weyertal 86-90, 50931 Cologne, Germany} 
\\
    \texttt{atattar@uni-koeln.de}
    \end{tabular}
    
\bigskip    
     \begin{tabular}{@{}l@{}}%
  \textsc{IMJ-PRG, Universit\'e Paris Cit\'e, B\^atiment Sophie Germain, 5 rue Thomas } \\ \textsc{Mann, 75205 Paris Cedex 13, France} 
\\
    \texttt{treffinger@imj-prg.fr}
    \end{tabular}

\end{document}